\documentclass[11pt]{article}

\usepackage{amsmath,amssymb,amsthm,url}

\theoremstyle{plain}      \newtheorem{theo}{Theorem}[section]
                          \newtheorem{pro}[theo]{Proposition}
                          \newtheorem{cor}[theo]{Corollary}
\theoremstyle{definition} \newtheorem{exa}{Example}[section]
                          
\theoremstyle{remark}	    \newtheorem*{rem}{Remark}

\date{\today}\numberwithin{equation}{section}

\usepackage{graphicx}

\usepackage{varioref}
\labelformat{section}{Section~#1}
\labelformat{subsection}{Section~#1}
\labelformat{figure}{Figure~#1}
\labelformat{table}{Table~#1}

\date{}

\newcommand\bbn{\mathbb{N}}
\newcommand\real{\mathbb{R}}
\newcommand\bbz{\mathbb{Z}}

\newcommand\hphi{\hat\phi}
\newcommand\phin{\phi^{(n)}}

\begin{document}

\title{Convolution powers of complex functions on $\mathbb Z$}

\author{Persi Diaconis\footnote{Supported in part by NSF grant 0804324.}\\
\textit{Department of Mathematics and Department of Statistics}\\\textit{Stanford University}\and%
        Laurent Saloff-Coste\footnote{Supported in part by NSF grant DMS 104771.}\ %
\footnote{Corresponding author: \texttt{lsc@math.cornell.edu}}\\
\textit{Department of Mathematics}\\\textit{Cornell University}}

\maketitle

\begin{abstract}
Repeated convolution of a probability measure on $\mathbb Z$ leads to the
central limit theorem and other limit theorems. This paper investigates
what kinds of results remain \textit{without} positivity.
It reviews theorems due to Schoenberg, Greville, and Thom\'ee which
are motivated by applications to data smoothing (Schoenberg and Greville)
and finite difference schemes (Thom\'ee).
Using Fourier transform arguments, we prove detailed  decay bounds
for convolution powers of  finitely supported
complex functions on $\mathbb Z$. If $M$ is an hermitian contraction,
an estimate for the
off-diagonal entries of the powers $M_k^n$ of $M_k=I-(I-M)^k$ is obtained.
This generalizes
the Carne--Varopoulos Markov chain estimate.
\end{abstract}

\section{Introduction} \label{sec-int}

This is the first of a series of papers regarding the following question.
What can be said about the convolution powers  of a finitely
supported function on a countable group?

A lot is known if we are willing to
assume that the function, say $\phi$,
is non-negative and normalized by $\sum \phi=1$. In this case, the convolution
power $\phin$ of $\phi$ is the probability distribution of the associated
random walk after $n$-steps and estimating $\phin$ is a well studied problem.
But what exactly is the role of ``positivity'' in classical results?

This article focuses mostly on the simplest case, convolutions on the integer
group $\mathbb Z$. It reviews what is known (and why people worked on this
question before) and provides some new results including an extension of a
Markov chain estimate due to Carne and Varopoulos. In forthcoming papers,
we will consider the case of
$\mathbb Z^d$, which is significantly different from the $1$-dimensional case,
and the case of non-commutative groups where completely different
techniques must be used and many open questions remain.

Let $\phi$ be a finitely supported probability measure on the integers $\bbz$.
Define convolution powers of $\phi$ by
\begin{equation*}
\phin(x)=\sum_y\phi^{(n-1)}(y)\phi(x-y).
\end{equation*}
The asymptotic ``shape'' of $\phin$ is described by the local
central limit theorem (assuming irreducibility and aperiodicity):
\begin{equation}
\phin(x)=\frac1{\sqrt{2\pi\sigma ^2n}}e^{-(x-\alpha n)^2/(2\sigma^2n)}+o\left(\frac1{\sqrt{n}}\right).
\label{11}
\end{equation}
In \eqref{11}, $\sigma^2=\sum x^2\phi(x)$, $\alpha=\sum x\phi(x)$,
and the error term is uniform over $x\in\bbz$. Proofs and many refinements are in \cite{bhatta,petrov}.

\begin{figure}[htb]
\centering
\includegraphics[keepaspectratio, width=.5\linewidth]{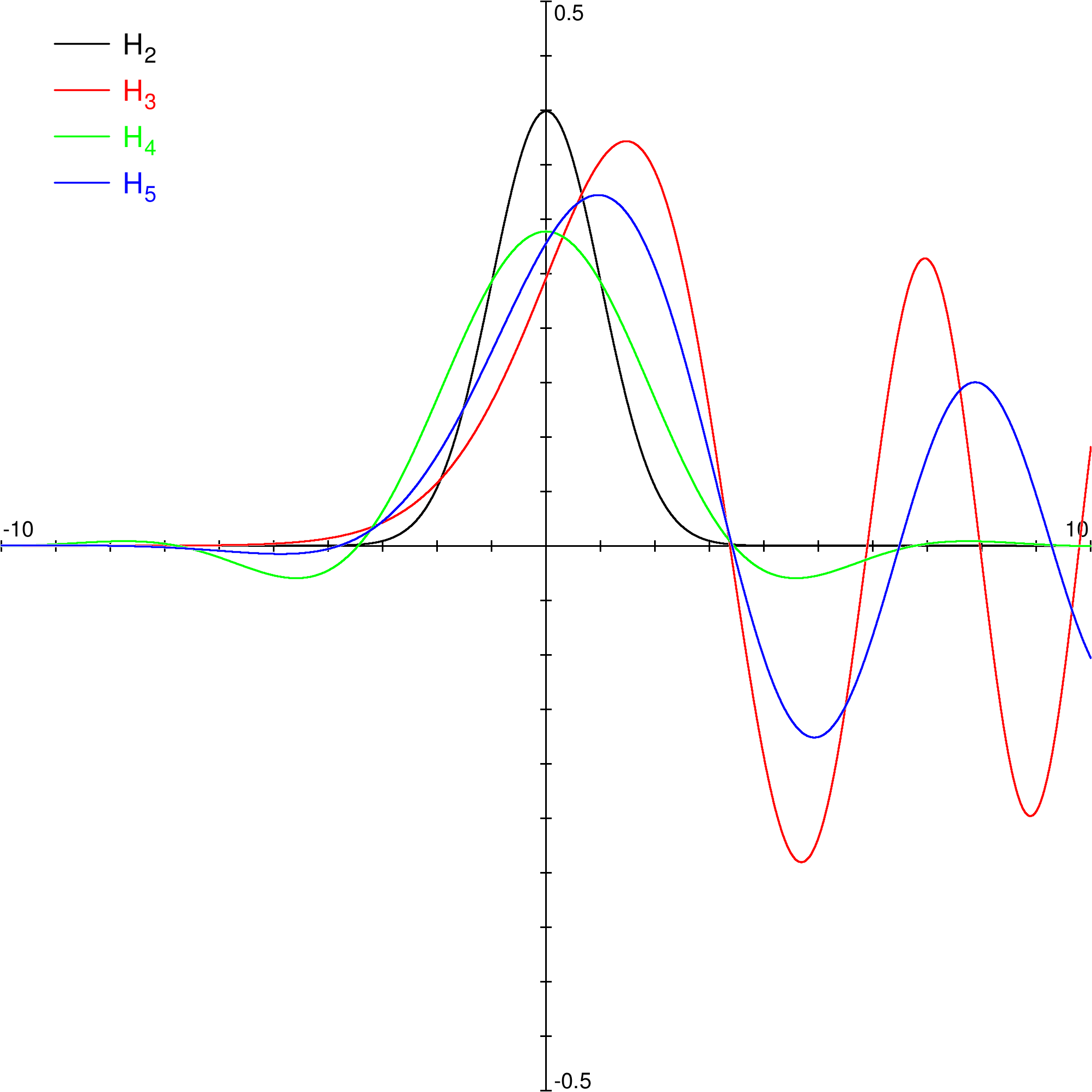}
\caption{}
\label{fig1}
\end{figure}

For this and other basic theorems of probability, could
there be similar results if $\phi(x)$ is a more general function?
For example, consider
\begin{equation}
\phi(-2)=\phi(2)=-\frac1{9},\quad\phi(1)=\phi(-1)=\frac49,\quad\phi(0)=\frac13\quad(\phi(x)=0\text{ otherwise}).
\label{12}
\end{equation}
The convolution powers $\phin$ are well-defined and results of Schoenberg \cite{Scho-Bull} imply
\begin{equation}
\phin(x)=\frac1{(n/9)^{1/4}}H_4\left(\frac{x}{(n/9)^{1/4}}\right)+o\left(\frac1{n^{1/4}}\right)
\label{13}
\end{equation}
with $H_4$ the real even function having Fourier transform $e^{-x^4}$.

More generally, $H_k$ is defined as having Fourier transform $e^{ -\xi^k}$ for even $k$ and  $e^{i\xi^k}$ for odd $k$.
 A graph of $H_4$ appears in \ref{fig1}. Also shown are the Gaussian density $H_2(x)=\frac1{\sqrt{4\pi}}e^{-x^2/4}$ and $H_3,\ H_5$. The function $H_3$ is the
famous Airy function (up to dilation) and it appears as the convolution limit of measures
such as
\begin{equation}
\phi(0)=1-3a,\quad\phi(1)=3a,\quad\phi(-1)=a,\quad\phi(2)=-a\quad(\phi(x)=0\text{ otherwise})
\label{14}
\end{equation}
for $0<a<1/4$ fixed. Careful statements of general theorems are in \ref{sec-SGT}.
We note that there are fundamental differences between the functions
$H_{2k}$ and the function $H_{2k+1}$. An exact formula exists only for $H_2$ but
$H_{2k}$ is even with $H_{2k}(0)>0$ and satisfies
\begin{equation}\label{H2k}
\forall x,\;\;|H_{2k}(x)|\le C_k\exp( -c_k|x|^{2k/(2k-1)}).
\end{equation}
In particular, $H_{2k}$ is in $L^1(\mathbb R)$. The functions $H_{2k+1}$ satisfy
$H_{2k+1}(0)>0$ but are not even, decay differently along the negative
and positive semi-axes
and are not absolutely integrable. For any $k>2$, $H_k$ changes sign infinitely many times.

The following theorem illustrates the main result of this
paper. It gives an exponential upper bound of the type (\ref{H2k})
for the convolution powers
of certain complex valued functions on $\bbz$. Such bounds are implicit in
the literature on finite difference methods. See \cite{Thomee1,Thomee2}.
The second part of the theorem gives lower bounds on the
real and imaginary parts in an appropriate neighborood of $0$.

\begin{theo}
Let $\phi(x)$ be a finitely supported complex-valued function on $\bbz$ with $\sum_x\phi(x)=1$. Assume that $\hphi(\theta)=\sum_x\phi(x)e^{ix\theta}$ satisfies
\begin{equation}
\left|\hphi(\theta)\right|<1\qquad\text{for }\theta\in(-\pi,\pi],\quad\theta\neq0.
\label{15}
\end{equation}
Assume further that there exist an even integer $\nu$ and a complex number $\gamma$ with $\mbox{\em Re}(\gamma)>0$ such that
\begin{equation}
\hphi(\theta)=e^{-\gamma\theta^\nu(1+o(1))}\qquad\text{as }\theta\to0.
\label{16}
\end{equation}
Then, there are constants $c,C\in(0,\infty)$ such that, for all $x\in\bbz,\ n\in\bbn^*$,
\begin{equation}\label{11up}
\left|\phin(x)\right|\leq\frac{C}{n^{1/\nu}}\exp\left(-c\left(\frac{|x|}{n^{1/\nu}}\right)^{\nu/(\nu-1)}\right).
\end{equation}
Further, there are constants $c_1,c_2>0$ such that, for all $x\in \mathbb Z, n\in \mathbb N$ with $|x|\le c_1n^{1/\nu}$, we have
\begin{equation*}
\mbox{\em Re}(\phin(x))\geq  c_2n^{-1/\nu}\end{equation*}
and
\begin{itemize}
\item if $\mbox{\em Im}(\gamma)\neq 0$,
$$|\mbox{\em Im}(\phi(x))|\ge c_2n^{-1/\nu}$$
with $\mbox{\em Im}(\phi^{(n)}(x))$ having the same sign as
$\mbox{\em Im}(\gamma)$;
\item if $\mbox{\em Im}(\gamma)=0$,
$$|\mbox{\em Im}(\phi(x)|)=o(n^{-1/\nu}).$$
\end{itemize}
\label{thm11}
\end{theo}

The first part of Theorem \ref{thm11} is proved in \ref{sec4}.
The second part follows from a local limit theorem derived in \ref{sec-ext}.
This local theorem gives further information about $\phin$ in the region
$|x|\le c_1n^{1/\nu}$. Each of these results is obtained  using Fourier
transform techniques
The conclusions of Theorem \ref{thm11} can be described roughly as follows.
Under the given hypothesis, $|\phi^{(n)}(x)|$ attains a  maximum of order
$1/n^{1/\nu}$ and is bounded below by $1/n^{1/\nu}$ in
a $n^{1/\nu}$-neighborhood of $0$. Further, as $x$ moves
away from $0$, $|\phin(x)|$ presents a relatively fast decay controlled  by
$$ \exp
\left(-c\left(\frac{|x|}{n^{1/\nu}}\right)^{\nu/(\nu-1)}\right).$$

In modern random walk theory, when $\phi$ is a centered probability
and $\nu$ can only take the value $\nu=2$,
the behaviors described in Theorem \ref{thm11} are often discussed under
the names of ``diagonal behavior''
(more precisely, ``near diagonal behavior'') and
``off-diagonal behavior''. In particular,
in the theory of random walks on graphs,
long range off-diagonal upper bounds of the type
$ e^{-cd(x,y)^2/n}$ are known as Carne--Varopoulos bounds
(here $d(x,y)$ denotes the
natural graph distance). In Section \ref{sec5},
we explain how Theorem \ref{thm11} leads to generalizations
of Carne-Varopoulos bounds.

In general,
it is not true that the $\ell_1$-norm of the convolution powers
$\phin$ of a complex valued function $\phi$ with $\sum \phi=1$
stay bounded uniformly in $n$.
However, as an immediate corollary of Theorem \ref{thm11} we recover
the following known result.
\begin{cor}Under the hypotheses of {\em Theorem \ref{thm11}}
there exists a constant $C$ such that
$$\forall\, n,\;\;\sum_{x\in \mathbb Z}|\phi^{(n)}(x)|\le C<\infty.$$
\end{cor}
This important property is studied in detail in
\cite{Aronson,Thomee1,Thomee2,wid65} because
of its close connections to the ``stability'' of certain approximation schemes
for partial differential equations. See Theorem \ref{oldthmA11}
below and the short discussion in \ref{sec22}.

The hypotheses made in Theorem \ref{thm11} are somewhat mysterious.
Indeed, it is very unclear how to replace these hypotheses in the context of
non-commutative countable groups (for which no viable Fourier transform exists,
in general). Even on $\mathbb Z$, these hypotheses
are not entirely natural since the Fourier transform $\hat\phi$ of a
function $\phi$ with $\sum_x\phi(x)=1$ may well attain its maximum
at multiple points and not at $0$. Sections 2 and 3 discuss this in detail.

We end this introduction with a brief description of the content of the paper.
Section 2 is devoted to local limit theorems for convolution of complex valued functions. In particular, \ref{sec-ext}
provides an extension of earlier results of Schoenberg and Greville
and \ref{sec-exa} describes illustrative examples. Section 3 is the main
section of this paper and  treats
upper bounds of the type (\ref{11up}).
See Theorems \ref{Gauss1} and \ref{Gauss4}.
In Section 4, the results of Section 3
and the discrete transmutation formula of Carne are used to obtain
long range estimates
for the powers of the operator  $M_k=I-(I-M)^k$ when $M$
is the infinite matrix of a reversible Markov chain
(more generally, an hermitian contraction).  Section 5 gives pointers
to various earlier works where convolutions of signed measures play an
explicit or implicit part.

\section{The results of Shoenberg, Greville and Thom\'e} \label{sec-SGT}
In this short section, we briefly review results by
Schoenberg, Greville and Thom\'ee. Schoenberg and Greville where motivated
by the earlier work of de Forest and statistical data smoothing procedures.
Thom\'ee's work is motivated by numerical approximation schemes
for differential equations. Brief explanations regarding the motivations of
these authors and other applications
are  collected in Section 5.

\subsection{Stability} The largest body of work concerning the problem discussed in this paper can be found in the literature concerning numerical approximation schemes for differential equations with work by  John \cite{john}, Aronson \cite{Aronson}, and Widlund \cite{wid65}.
 Thom\'e's articles \cite{Thomee1,Thomee2} give excellent pointers to the literature. However, these papers do not isolate or emphasize the basic convolution aspect
of the results. This makes extracting the relevant results somewhat difficult. Nevertheless, the following essential result is explicitly stated in \cite{Thomee2}. Given a complex valued absolutely summable function $\phi$ on $\mathbb Z$, consider the property
\begin{equation}
\forall\,n,\qquad\sum_x|\phi^{(n)}(x)|\leq C<\infty.
\label{thomee}
\end{equation}
\begin{theo}[{See \cite[Theorem 7.54]{Thomee2}}]
Condition \eqref{thomee} is satisfied if and only one or the other of the following two conditions is satisfied:
\begin{enumerate}
\item $\hphi(\xi)=\zeta e^{iy\xi}$ for some $y\in\mathbb Z$ and $\zeta\in\mathbb C$ with $|\zeta|=1$
\item $|\hphi(\xi)|<1$ except for at most a finite number of points $\xi_q$, $q\in\{1,\dots,Q\}$ in $|\xi|\leq\pi$ where $|\hat a(\xi_q)|=1$, and there are constants $\alpha_q,\beta_q,\nu_q$, $q=1,\dots,Q$, with $\alpha_q$ real, Re($\beta_q)>0$ and $\nu_q$ an even natural integer, such that
\begin{equation*}
\hphi(\xi_q+\xi)=\hphi(\xi_q)
\exp\left(i\alpha_q\xi-\beta_q\xi^{\nu_q}\left(1+o(1)\right)\right)\qquad\text{as }\xi\to0.
\end{equation*}
\end{enumerate}
\label{oldthmA11}
\end{theo}
Thom\'ee also describes what happens when the principal nonlinear term of
the expansion at a point is imaginary (of odd or even degree).
In particular, if $|\hphi|\le 1$, $\hphi(0)=1$ and $0$ is  the unique point where $|\hphi|$ is maximum with
\begin{equation*}
\hphi(\xi)=\exp\left(i\alpha\xi+i\gamma\xi^\mu-\beta\xi^\nu\left(1+o(1)\right)\right)\qquad\text{as }\xi\to0
\end{equation*}
with $\alpha,\gamma$ real, $\gamma\neq0$, Re$(\beta)>0$, $0<\mu\leq\nu$ and $\nu$ an even integer, \cite[(7.9)]{Thomee2} gives
\begin{equation}\label{Thomest}
\sum_y|\phi^{(n)}(y)|\simeq n^{(1-\mu/\nu)/2}.
\end{equation}
Here, $\simeq $ means that the ratio of the two sides stays between two positive constants as $n$ tends to infinity.

Condition \eqref{thomee} is called a stability condition. In the case of smoothing, it shows that the iterated smoother is continuous from $\ell^\infty$ to $\ell^\infty$; small changes in the input sequence lead to close smoothed sequences, uniformly in the iteration number $n$. In the case of divided difference schemes, it shows that the scheme applied to bounded data gives bounded output, uniformly in time.

\subsection{De Forest local limit theorems}
Following de Forest, Schoenberg \cite{Scho-Bull} and Greville \cite{Grev2} explore local limit theorems
for real valued  $\phi$ that are allowed to change sign and are normalized by
$\sum\phi=1$. Schoenberg treats the case of symmetric $\phi$
(i.e., $\phi(x)=\phi(-x)$) under the assumption that
$|\hphi(\theta)|<1$ for $\theta\in (0,2\pi)$ and
$\phi(\theta)=1-\lambda \theta^{k}+0(|\theta|^{2k+1})$
(with, by necessity, $\lambda>0$).  Greville observes that symmetry is not
essential and also treats the case when $\hphi(\theta)=1+a i\theta^{2k+1}+
O(|\theta|^{2k+2}).$

For each positive integer $k$, let
$H_k$ denote the function of one real variable defined by
$\hat{H}_k(\xi)=e^{ -\xi^k}$ for even $k$ and $\hat{H}_k(\xi)=e^{i\xi^k}$ for odd $k$.

\begin{theo}[Greville \cite{Grev2} and Schoenberg \cite{Scho-Bull}, in the spirit of de Forest]
Let $\phi$ be a real summable function on $\mathbb Z$ with Fourier transform $\hphi$ satisfying $|\hphi(\theta)|<1$ for all $\theta\in(0,2\pi)$ and
\begin{equation}
\hphi(\theta)=1+a(i\theta)^m+O(|\theta|^{m+1})\qquad\text{with }0\neq a\in \mathbb R \mbox{ and } m \mbox{ an integer greater than }1.
\label{hyp}
\end{equation}
\begin{itemize}
\item If $m$ is even then $-ai^m=\lambda$ must be positive and we have
\begin{equation*}
\phin(x)=(\lambda n)^{-1/m}H_m\left( x(\lambda n)^{-1/m}\right)+o(n^{-1/m}),
\end{equation*}
uniformly over $x\in \mathbb Z$.
\item If $m$ is odd and $\epsilon$ denotes the sign of the real $ai^{m-1}$, $\epsilon=\emph{sign}(ai^{m-1})$, then we have
\begin{equation*}
\phin(x)=(|a|n)^{-1/m}H_m\left(\epsilon x(|a| n)^{-1/m}\right)+o(n^{-1/m}).
\end{equation*}
\end{itemize}
\label{oldthmB23}
\end{theo}
\begin{rem} If we consider the expansion $\hphi(\theta)=\sum_j A_j (i\theta)^j$
then
$A_j=\frac{1}{j!}\sum_x\phi(x)x^j$. Hypothesis (\ref{hyp}) is simply the
assumption that the first $k-1$ moments of $\phi$ vanish
and $a=A_k\neq 0$. Note that $a$ is real here because $\phi$ is real.
\end{rem}
\begin{rem}The hypotheses made in Theorem \ref{oldthmB23}
are somewhat ad hoc. Schoenberg and Greville are interested in data smoothing procedures. This explains the basic hypotheses that $\phi$ is real and satisfies $\sum \phi=1$ but these properties are not essential, at least in the case when $m$ is even. The assumption that $\max|\hphi|= 1$ is actually the essential condition. Coupled with the condition $\sum\phi=1$, it implies that
$|\hphi|$ attains its maximum at $\theta=0$ and that, in fact, $\hphi(0)=1$. Schoenberg and Greville then make the additional assumptions that $|\hphi(\theta)|<1$ if $\theta\neq 0$. The condition (\ref{hyp}) simply captures the vanishing order of $\hphi -1$ at $0$. It excludes the possibility of a drift. In the next section, we extend the even $m$ result to allow  complex $\phi$, a drift, and multiple points where $|\hat\phi|$ attains it maximum.
\end{rem}

\begin{rem}  The proof of Theorem \ref{oldthmB23} is significantly more difficult in the case
where $m$ is odd than in the case when $m$ is even. Indeed, for odd $m=2k+1>1$,
\begin{equation*}
H_{2k+1}(x)=\frac1{\pi}\int_0^\infty\cos(u^{2k+1}-xu)\ du.
\end{equation*}
This integral is not absolutely convergent  and $H_{2k+1}$ is neither integrable nor square integrable. Nevertheless, it is an entire function.
\end{rem}

\subsection{An extended local limit theorem in the even case}\label{sec-ext}
Theorem \ref{oldthmB23} is taken from \cite{Grev2}. Its form is somewhat constrained, perhaps due to the applications that Greville had in mind. Comparison with Theorem \ref{oldthmA11} leads to the following extension where
$H_{2k,b}$ denotes the complex valued Schwartz function
whose Fourier transform is $\hat{H}_{2k,b}(\xi)=e^{-(1+ib)\xi^{2k}}$, $b\in \mathbb R$.

\begin{theo}\label{th-extLL}
Let $\phi$ be a complex absolutely summable function on $\mathbb Z$ with Fourier transform $\hphi$ satisfying $|\hphi(\theta)|\le 1$ for all $\theta\in [-\pi,\pi]$. Assume that $|\hphi(\xi)|<1$ except for at most a finite number of points $\xi_q$, $q\in\{1,\dots,Q\}$, in $|\xi|\leq\pi$ where $|\hphi(\xi_q)|=1$. Assume further that there is an integer $Q_1\le Q$, an even integer $m$ and constants $\alpha_q,\beta_q=b_q+ib'_q$, $q=1,\dots,Q_1$, with $\alpha_q,b_q,b'_q\in \mathbb R$, $b_q=\mbox{\em Re}(\beta_q)>0$ such that
\begin{equation}\label{hypq}
\forall\, q\in \{1,\dots,Q_1\},\;\;
\hphi(\xi_q+\xi)=\hphi(\xi_q)
\exp\left(i\alpha_q\xi -\beta_q\xi^{m}\left(1+o(1)\right)\right)\qquad\text{as }\xi\to 0;
\end{equation}
and
\begin{equation*}\forall\, q\in \{Q_1+1,\dots,Q\},\;\;
|\hphi(\xi_q+\xi)|\le
\exp\left( -c_q|\xi|^{\mu}\right)
\qquad\text{as }\xi\to 0,
\end{equation*}
for some $\mu\in (0,m)$.
Then we have
\begin{equation*}
\phin(x)=\sum_1^{Q_1} (b_qn)^{-1/m}e^{-ix\xi_q}\hphi(\xi_q)^nH_{m,b'_q/b_q}((x-\alpha_q n)(b_qn)^{-1/m})
+o(n^{-1/m}),
\end{equation*}
where the error term is uniform in $x\in \mathbb Z$.
\label{thmGG}
\end{theo}
\begin{rem}To make sense of this result, it is necessary to know something about
the functions $H_{2k,\beta}$. By definition,
$$\mbox{Re}(H_{2k,b}(x))=\frac{1}{\pi}\int_0^\infty
(\cos x\xi)(\cos (b\xi^{2k}))e^{-\xi^{2k}}d\xi $$ and
$$\mbox{Im}(H_{2k,b}(x))=\frac{1}{\pi}\int_0^\infty
(\cos x\xi)(\sin (b\xi^{2k}))e^{-\xi^{2k}}d\xi. $$
From these formula, one easily sees that
$\mbox{Re}(H_{2k,b}(0))$ is always positive and that
$\mbox{Im}(H_{2k,b}(0))$ is zero, positive or negative if and only if $b$
is zero, positive or negative.
Further, it holds that
$$|H_{2k,b}(x)|\le C_{k,b} \exp(-c_{k,b}|x|^{\frac{2k}{2k-1}}).$$
In the case $k=1$, we have
$$H_{2,b}(x)= \frac{1}{\sqrt{4\pi (1+ib)}}e^{-\frac{|x|^2}{4(1+ib)}}.$$
This complex valued function of the real variable $x$ is the heat
kernel, i.e., the kernel of $e^{z \Delta}$,
computed at the complex time $z=1+ib$. Here $\Delta$ denotes the unique
self-adjoint extension of $(d/dx)^2$ originally defined
on smooth compactly supported functions.
\end{rem}

\begin{proof}
Set $I=(-\pi,\pi]$. For each $q$, let $I_q=[\xi_q-\epsilon,\xi_q+\epsilon]$, $\epsilon>0$, be a small interval centered around $\xi_q$ and let $J=I\setminus \cup_1^Q I_q$. We have
\begin{eqnarray*}2\pi\phi^{(n)}(x)&= &
\int_{-\pi}^{\pi} e^{-i x\theta} [\hphi(\theta)]^n d\theta\\
&=& \int_J e^{-i x\theta} [\hphi(\theta)]^n d\theta+
\sum_{q=1}^Q \int_{I_q}e^{-i x\theta} [\hphi(\theta)]^n d\theta
=\sum_0^Q \Phi_q
\end{eqnarray*}
where
$$\Phi_0=\int_J e^{-i x\theta} [\hphi(\theta)]^n d\theta
\;\mbox{ and }\;
\Phi_q=\int_{I_q}e^{-i x\theta} [\hphi(\theta)]^n d\theta,\; q=1,\dots, Q.$$
On $J$, there exists $\rho=\rho_\epsilon\in (0,1)$ such that
$|\hphi|\le \rho$ and thus, $|\Phi_0|=O(\rho^n)$.
For $q\in \{Q_1+1,\dots,Q\}$, we have
\begin{eqnarray*}
|\Phi_q|&=& |\int_{-\epsilon}^{\epsilon}e^{-i x(\xi_q+\theta)} [\hphi(\xi_q+\theta)]^n d\theta|\\
&\le&
\int_{-\epsilon}^{\epsilon} e^{-c_q n|\theta|^{\mu}} d\theta\\
&\le& n^{-1/\mu}\int_{-\infty}^\infty
e^{-c_q|u|^\mu}du= o(n^{-1/m}).
\end{eqnarray*}
The main contribution comes from the integrals $\Phi_q$, $q\in \{1,\cdots, Q_1\}$. We set
$$\psi_{q,n}(u)=\hphi(\xi_q)^{-1}e^{-\alpha_q u (b_qn)^{-1/m}}\hphi(\xi_q+u(b_q n)^{-1/m})]^n,$$
$$y_{q,n}=\frac{x-\alpha_qn}{(b_qn)^{1/m}},$$
and   write
\begin{eqnarray*}
\Phi_q &=& e^{-ix\xi_q}\hphi(\xi_q)^n\int_{-\epsilon}^\epsilon e^{-i \frac{(x-\alpha_qn)}{(b_q n)^{1/m}} [(b_q n)^{1/m}\theta]} [\hphi(\xi_q)^{-1}e^{-i\alpha_q\theta}\hphi(\xi_q+\theta)]^n d\theta\\
&=& (b_q n)^{-1/m} e^{-ix\xi_q}\hphi(\xi_q)^n\int_{-(b_qn)
^{1/m}\epsilon}^{(b_qn)
^{1/m}\epsilon} e^{-i y_{q,n}u} \psi_{q,n}(u) du.
\end{eqnarray*}
Next, since
$$2\pi H_{m,b'_q/b_q}(y_{q,n})=\int_{-\infty}^{\infty}e^{-y_{q,n}u}
e^{-(1+ib'_q/b_q)u^m}du.$$
we have
\begin{eqnarray*}\lefteqn{\Phi_q - 2\pi
(b_q n)^{-1/m} e^{-ix\xi_q} \hphi(\xi_q)^nH_{m,b'_q/b_q}(y_{q,n}))=}\vspace{.5in}&&\\
&&\int_{(b_qn)^{1/m}I_q}e^{-i y_{q,n}u} \psi_{q,n}(u) du -\int_{-\infty}^{\infty}e^{-iy_{q,n}u}
e^{-(1+ib'_q/b_q)u^m}du\\
&=&\int_{|u|\le \epsilon(b_qn)^{1/m}}e^{-i y_{q,n}u} [\psi_{q,n}(u)
-e^{-(1+ib'_q/b_q)u^m}]du\\&& \vspace{.5in}-\int_{|u|>\epsilon(b_q n)^{1/m}}e^{-iy_{q,n}u}
e^{-(1+ib'_q/b_q)u^m}du =\mathcal I_1-\mathcal I_2.
\end{eqnarray*}
The integral $\mathcal I_2$ can be estimated brutally by
$$|\mathcal I_2|\le 2\int_{u>\epsilon(b_q n)^{1/m}}
e^{-u^m}du= O(e^{-\epsilon^mb_q n}).$$
To estimate $\mathcal I_1$, note that (\ref{hypq}) shows that for any $\eta>0$ there exists $\epsilon>0$ such that
$$|\psi_{q,n}(u)-e^{-(1+ib'_q/b_q)u^m}|\le e^{-u^m}|e^{\eta u^m }-1|\le C\eta e^{-u^m/4}$$
for all $|u|\le \epsilon(b_q n)^{1/m}.$ It follows that
$$|\mathcal I_1|\le 2C\eta \int_0^\infty e^{-u^m/4}du.$$
Putting all these estimates together, we find that
$$ \phi^{(n)}(x)- (b_q n)^{-1/m} e^{-ix\xi_q} \hphi(\xi_q)^nH_{m,b'_q/b_q}(y_{q,n}))
= o(1/n^{1/m}).$$
\end{proof}

\subsection{Examples}\label{sec-exa}
 One of the simplest example that can be used to illustrate the results discussed in this paper appears in the following proposition.
\begin{pro}\label{prop-B}
Assume that $\phi$ is real symmetric, $|\hphi|\le 1$
and there exists $a>0$ such that
\begin{equation} \label{26}
\hphi(\theta)=1-a
\theta^{2k}(1+o(1)) \mbox{ at }
\theta=0. \end{equation}
Then  $\phi(x)\neq 0$ for some $x\ge k$.
If we assume further that $\phi$ is supported on $\{-k,\dots,k\}$ then
$a\in (0,2^{-2k+1}]$ and
\begin{eqnarray}
\phi &=& \delta_0 -\lambda (\delta_0-\beta)^k
= \delta_0 -(\delta_0-\beta_{\lambda^{1/k}})^k \label{phi=B}\\
&=&\sum_{j=1}^k (-1)^j\binom k j\beta^{(j)}_{\lambda^{1/k}} \nonumber
\end{eqnarray}
where $\lambda=a2^{k}\in (0,2^{-k+1}]$, $\beta=\frac{1}{2}(\delta_{-1}+\delta_1)$ and
 $\beta_s= (1-s)\delta_0+s\beta$, $s\in (0,1)$.
\end{pro}
\begin{rem}The function $\phi$ defined at (\ref{phi=B}) satisfies
$$\hphi(\theta)=1-\lambda(1-\cos\theta)^k.$$
It follows that $\max\{|\hphi|\}=1$ as well as (\ref{26}) for any
$\lambda=a2^{k}\in (0,2^{-k+1}]$.The maximum $1=\max\{|\hphi|\}$ is attained
solely at $0$ if and only if $\lambda \in (0,2^{-k+1})$.
Note that any parameter $\lambda$ in the range $(0,1/2]$ is admissible
for all values of $k$.
\end{rem}
\begin{rem} If convolution by $\delta_0-\beta_s$ (for some fixed $s\in (0,1)$, say $s=1/2$) is interpreted as the discrete analog of the (positive) Laplacian $-\partial^2_x$ then convolution by $\delta_0-\phi= (\delta_0-\beta_s)^k$
is analogous to the higher even powers of the Laplacian, that is, $(-1)^k\partial^{2k}_x$.
\end{rem}
\begin{proof} By assumption, if $p$ is the largest integer such that $\phi(p)\neq 0$, $\hphi$ is a polynomial $Q$
in $\cos\theta$ of degree $p$. If \eqref{26} holds then the polynomial $1-Q$ vanishes of order $k$ at $1$. Hence, we must have $p\ge k$. If we assume that $\phi$ is supported on $\{-k,\dots,k\}$, then we must must have $1-Q(u)=\lambda(1-u)^k$ and $a=2^k\lambda$ with $a$ as in (\ref{26}). That is,
\begin{equation*}
\hphi(\theta)=1-\lambda(1-\cos\theta)^k=
1-\left(1-\left(1-\lambda^{1/k}\right)+\lambda^{1/k}\cos\theta\right)^k.
\end{equation*}
The condition that $|\hphi|<1$ on $(0,2\pi)$ translates into $1-\lambda 2^k>-1$, that is $\lambda<2^{-k+1}$.
Further, since $\cos\theta=\hat\beta(\theta)$ where $\beta(1)=\beta(-1)=1/2$ and $\beta=0$ otherwise (i.e., $\beta=$Bernoulli($1/2$)), we see that we must have
\begin{equation*}
\phi=-\sum_{i=1}^k(-1)^i\binom{k}{i}\beta_{\lambda^{1/k}}^{(i)}
\end{equation*}
where $\beta_s=(1-s)\delta_0+s\beta$.
\end{proof}

\begin{exa} In Proposition \ref{prop-B}, consider the case when $k=1$ and $\lambda=1$ so that $\phi= \beta$.
We have $\hphi(\theta)=\cos\theta$. In this classical case, Theorem \ref{th-extLL} yields
$$\phi^{(n)}(x) = (1+(-1)^ne^{-ix\pi})(n/2)^{-1/2}H_2(x/(n/2)^{1/2})+o(n^{-1/2}).$$
This captures the periodicity of the Bernoulli walk.
\end{exa}

\begin{exa} In Proposition \ref{prop-B}, consider the case when $k=2 $, that is
$\phi= 2\beta_s-\beta_s^{(2)}$ with $s\in (0,1/\sqrt{2})$.
We have $\hphi(\theta)=1-s^2(1-\cos\theta)^2$. Theorem \ref{oldthmB23} yields
$$\phi^{(n)}(x) \sim  (4/s^2n)^{1/4}H_4(x(4/(s^2n)^{1/4}).$$
The same result holds true in the limit case where $s=1/\sqrt{2}$ and $ \phi=2\beta_{1/\sqrt{2}}-\beta_{1/\sqrt{2}}^{(2)}$. However,
in this case,  $|\hphi|\leq 1$ and $\hphi(\theta)=1$ if and only if $\theta=0$ or $\theta=\pi$. At $0$, $\phi(\theta)=1-\frac18|\theta|^4+O(|\theta|^6)$. At $\pi$, $\hphi(\theta)=-1+(\pi-\theta)^2+O(|\pi-\theta|^4)$. To obtain the desired asymptotic, apply
Theorem \ref{th-extLL}. Compare to the previous example.
\label{exB33}
\end{exa}

\begin{exa} Let $\phi$ be defined by $\phi(0)=5/8$, $\phi(\pm 2)=-1/4$,
$\phi(\pm 4)=-1/16$ and $\phi(x)=0$ otherwise. We have
\begin{eqnarray*}
\hphi(\theta)&= &\frac{5}{8}- \frac{1}{2} \cos2\theta -\frac{1}{8}\cos 4 \theta\\
&=& \frac{5}{4} -\cos^2\theta -\frac{1}{4}\cos^2 2\theta \\
&=& 1   -\cos^4\theta.
\end{eqnarray*}
Hence $|\hphi|\le 1$ and $\hphi(\theta)=1$ if and only if $\theta=\pm \pi/2$.
Further, at $\theta_\pm=\pm \pi/2$, $$\hphi(\theta)= 1-|\theta-\theta_\pm|^4+O(|\theta-\theta_\pm|^5).$$
Theorem \ref{th-extLL} applies and gives
\begin{eqnarray*}
\phi^{(n)}(x)&=& (e^{-ix\frac{\pi}{2}}+e^{ix\frac{\pi}{2}})n^{-1/4}H_4(xn^{-1/4})+o(n^{-1/4})\\
&=&2\cos (\pi x/2)n^{-1/4}H_4(xn^{-1/4})+o(n^{-1/4}).\end{eqnarray*}
\end{exa}

\begin{exa} \label{3points} In this example, we consider the convolution powers of an arbitrary real valued function supported on $\{-1,0,+1\}$
(except for some trivial cases). For $a_0,a_+,a_-\in \mathbb R$,
let $\phi $ be given by
$$\phi(0)=a_0, \;\;\phi(\pm 1)=a_\pm \mbox{ and }\phi=0 \mbox{ otherwise}.$$
We assume $a_0>0$ and that either $a_+\neq 0$ or $a_-\neq 0$
to avoid trivialities. We do not assume any normalization. In particular,
$$\sum \phi=a_0+a_++a_-$$
is an arbitrary real number. The following proposition shows that there are
essentially 3 different ``generic'' cases (each occurring on an open
subset of the parameter space). In only one of these 3 cases is the
normalization $\sum \phi=1$ the correct normalization giving
$\max \{|\hphi|\}= 1$. In the other two cases, different normalizations
are needed to insure $\max \{|\hphi|\}= 1$. Because we do not incorporate
any normalization, the asymptotic described below for $\phi^{(n)}$
contain an exponential term $A^n$. In each case, the constant $A$ satisfies
$A=|\max\{|\hphi|\}$ and is given explicitly.

\begin{pro} Referring to the function $\phi$ defined above, the following asymptotics hold true:
\begin{itemize}
\item Assume that $a_+a_-\ge 0$ or that $a_+a_-<0$ and $4|a_+a_-|< a_0|a_++a_-|$.
Set $$A=a_0+|a_+|+|a_-|, \;\;\alpha= \frac{a_+-a_-}{A} \;\mbox{ and }\;
\gamma=\frac{|a_+|+|a_-|}{2A} -\frac{\alpha^2}{2}.$$ Then
$$\phi^{(n)}(x)=\left(\frac{a_++a_-}{|a_++a_-|}\right)^x A^n (\gamma n)^{-1/2}
e^{-|x-\alpha n|^2/\gamma n} +o(A^nn^{-1/2})$$
where the error term is uniform in $x$.
\item Assume that $a_+a_-<0$ and $4|a_+a_-|> a_0|a_++a_-|$. Set
$$A=|a_+-a_-|(1+a_0^2/4|a_+a_-|)^{1/2}, \;\;\alpha=\frac{a_++a_-}{a_+-a_-}.$$
Let $\theta_0\in (0,\pi)$ be defined by $\cos\theta_0=-a_0(a_++a_-)/4a_+a_-$
and set
$$
b=
\frac{4a_0|a_+a_-|\sin\theta_0}{(a_+-a_-)(4|a_+a_-|+a_0^2)},\;\;
\gamma
= \frac{16|a_+a_-|^2-a_0^2(a_++a_-)^2}{2(a_0^2+4|a_+a_-|)(a_+-a_-)^2}.$$
Let $\omega_0$ be the argument of $\hphi(\theta_0)$. Then
\begin{eqnarray*}
\phi^{(n)}(x)
&=& (n/\gamma)^{-1/2}A^n
e^{-ix\theta_0+in\omega_0}H_{2,b/\gamma}((x-\alpha n)/(\gamma n)^{1/2})\\
&&+ (n/\gamma)^{-1/2} A^n
e^{ix\theta_0-in\omega_0}H_{2,-b/\gamma}((x-\alpha n)/(\gamma n)^{1/2})
+o(A^n n^{-1/2})
\end{eqnarray*}
where the error term is uniform in $x$.
\end{itemize}
\label{pro-exa}
\end{pro}
\begin{rem} The two principal terms on the right-hand side of the last equation
are complex conjugate so that their  sum is real (as it should be since $\phi$
is real valued).
\end{rem}
\begin{rem}This example can also be used to illustrate the stability theorem,
Theorem \ref{oldthmA11}. Indeed, all cases
with $a_0>0$ and either $a_+$ or $a_-$ non-zero are considered in Proposition
\ref{pro-exa} except for the very special case
when $4|a_+a_-|=a_0|a_++a_-|$. Let $A=A(a_0,a_+,a_-)$ be as defined
in Proposition \ref{pro-exa}. As a corollary of the proof given below,
it follows that the normalized function $\phi_0=\phi/A$ satisfies  the
stability condition (\ref{thomee}) in all cases but the special case
$4|a_+a_-|=a_0|a_++a_-|$ for which it actually fails.
\end{rem}
\begin{proof}
Obviously
$$\hphi(\theta)= a_0+ (a_++a_-)\cos \theta+ i(a_+-a_-)\sin \theta$$
and
$$|\hphi(\theta)|^2= a_0^2+a_+^2+a_-^2 +2a_+a_-(2\cos^2 \theta -1)+ 2a_0(a_++a_-)\cos\theta.$$
Further
$$\hphi(0)=a_0+a_++a_-,\;\ \hphi(\pi)= a_0 -(a_++a_-).$$
If $a_+a_-\ge 0$ then $|\hphi|$ has a maximum which is attained only at
$0$ if $a_++a_-> 0$ and only at $\pi$ if $a_++a_-<0$.

Assume first that $a_+a_-\ge 0$ and set
$$A=a_0+|a_+|+|a_-|, \;\;\alpha=\frac{a_+-a_-}{A} \mbox{ and }\gamma=
\frac{|a_+|+|a_-|}{2A}-\frac{(a_+-a_-)^2}{2A^2}>0.$$
Considering separately the two cases  $a_++a_->0$ and $a_++a_-<0$, we obtain
\begin{eqnarray*}
e^{-i\alpha\theta}\hphi(0)^{-1}\hphi(\theta)&=& (1- i\alpha \theta -\frac{\alpha^2}{2}\theta^2
+o(|\theta|^2))(1 +i\alpha \theta- (\gamma+\alpha^2/2)\theta^2 +o(|\theta|^2))\\
&=& 1-\gamma\theta^2(1+o(1))=e^{-\gamma \theta^2(1+o(1))}.
\end{eqnarray*}
Hence, if $a_++a_->0$ we have
$$\phi^{(n)}(x)= A^n (\gamma n )^{-1/2}H_2((x-\alpha n)/(\gamma n)^{1/2})+o(A^{n}n^{-1/2}).$$
If instead $a_++a_-<0$ then we have
$$\phi^{(n)}(x)= (-1)^x A^n (\gamma n )^{-1/2}H_2((x-\alpha n)/(\gamma n)^{1/2})+o(A^{n}n^{-1/2}).$$

Next we consider what happens when $a_+a_-<0$. Computing the derivative
of $f(\theta)=|\hphi(\theta)|^2$ gives
$$f'(\theta)=-2( 4a_+a_- \cos \theta + a_0(a_++a_-))\sin \theta.$$

If $4|a_+a_-|< a_0|a_++a_-|$, then $|\hphi|^2$ attains its maxima at $0$
if $a_++a_->0$ and at $\pi$ if $a_++a_-<0$. In each case, the asymptotic is the same as described above.

If $4|a_+a_-|\ge a_0|a_++a_-|$, then we set $\theta_0 $ to be the point in $(0,\pi)$ such that
$$\cos \theta_0= -\frac{a_0(a_++a_-)}{4a_+a_-}$$ and note that
$|\hphi|^2$ has twin maxima at $\theta=\pm \theta_0$ where
\begin{eqnarray*}
A^2&=&|\hphi(\pm\theta_0)|^2= a_0^2+a_+^2+a_-^2 +2|a_+a_-| +a_0^2\frac{(a_++a_-)^2}{4|a_+a_-|}\\
&=&(a_+-a_-)^2\left(1+\frac{a_0^2}{4|a_+a_-|}\right) .\end{eqnarray*}
Further
\begin{eqnarray*}
\mbox{Re}(\overline{\hphi(\theta_0)}\hphi(\theta))
&=& (a_0+ (a_++a_-)\cos \theta_0)(a_0+ (a_++a_-)\cos \theta)
+ (a_+-a_-)^2\sin \theta_0\sin \theta \\&=&
(a_+-a_-)^2\left(\frac{a^2_0}{4|a_+a_-|}+ \cos\theta_0 \cos\theta + \sin \theta_0\sin \theta\right)
\\&=&
(a_+-a_-)^2\left(\frac{a^2_0}{4|a_+a_-|}+ \cos (\theta_0-\theta)\right)\\
&=&
|\hphi(\theta_0)|^2 \left(1- \frac{2|a_+a_-|}{a_0^2+4|a_+a_-|}(\theta-\theta_0)^2(1+o(1))\right)
\end{eqnarray*}
and
\begin{eqnarray*}
\mbox{Im}(\overline{\hphi(\theta_0)}\hphi(\theta))
&=&  (a_+-a_-)( (a_0+(a_++a_-)\cos\theta_0)\sin\theta -(a_0+(a_++a_-)\cos\theta) \sin \theta_0 )\\
&=&
a_0(a_+-a_-)(\sin \theta -\sin \theta_0)+ (a^2_+-a^2_-)\sin (\theta-\theta_0)\\
&=& (a_+-a_-)(a_0 \cos \theta_0 +a_++a_- )(\theta-\theta_0)\\
&& -
\frac{a_0(a_+-a_-)}{2}\sin\theta_0 (\theta -\theta_0)^2 (1+o(1))\\
&=& (a^2_+-a^2_-)\left(1+\frac{a_0^2}{4|a_+a_-|} \right)(\theta-\theta_0)\\
&& -
\frac{a_0(a_+-a_-)}{2}\sin\theta_0 (\theta -\theta_0)^2 (1+o(1))\\
&=&|\hphi(\theta_0)|^2
\frac{a_++a_-}{a_+-a_-}(\theta-\theta_0) -\\
&& |\hphi(\theta_0)|^2
\frac{4a_0|a_+a_-|\sin\theta_0}{(a_+-a_-)(4|a_+a_-|+a_0^2)} (\theta -\theta_0)^2 (1+o(1))
\end{eqnarray*}

Set
$$\alpha = \frac{a_++a_-}{a_+-a_-}
,\;\;b=
\frac{4a_0|a_+a_-|\sin\theta_0}{(a_+-a_-)(4|a_+a_-|+a_0^2)}$$
and
\begin{eqnarray*}
\gamma&=& \frac{2|a_+a_-|}{a_0^2+4|a_+a_-|}
 -\frac{1}{2}\left(\frac{a_++a_-}{a_+-a_-}\right)^2\\
&=& \frac{16|a_+a_-|^2-a_0^2(a_++a_-)^2}{2(a_0^2+4|a_+a_-|)(a_+-a_-)^2}.
\end{eqnarray*}
Note that $\gamma$ is (strictly) positive
if $4|a_+a_-|>a_0|a_++a_-|.$
With this notation, assuming that
$4|a_+a_-|>a_0|a_++a_-|$, we have
$$\hphi(\theta_0+\theta)=\phi(\theta_0)e^{i\alpha \theta -(\gamma+i\beta)\theta^2(1+o(1))}.$$
Similarly,
$$\hphi(-\theta_0+\theta)=\phi(-\theta_0)e^{i\alpha \theta -(\gamma-i\beta)\theta^2(1+o(1))}.$$
Let $\omega_0$ be the argument of $\hphi(\theta_0)$. Then
Theorem \ref{th-extLL} gives
\begin{eqnarray*}
\phi^{(n)}(x)
&=& (n/\gamma)^{-1/2}A^n
e^{-ix\theta_0+in\omega_0}H_{2,b/\gamma}((x-\alpha n)/(\gamma n)^{1/2})\\
&&+ (n/\gamma)^{-1/2} A^n
e^{ix\theta_0-in\omega_0}H_{2,-b/\gamma}((x-\alpha n)/(\gamma n)^{1/2})
+o(^n n^{-1/2}).
\end{eqnarray*}

Finally, if $4a_+a_-=-a_0(a_++a_-)$ (resp. $4a_+a_-=a_0(a_++a_-)$), $|\hphi|$
is maximum at $0$ (resp. at $\pi$). The two case are similar and we treat
only the case when $4a_+a_-=-a_0(a_++a_-)$. In this case  we have
$$\hphi(0)^{-1}\hphi(\theta)=e^{i\alpha\theta-i\frac{1}{6}(\alpha-\alpha^3)\theta^3
-\frac{1}{8}(\alpha^2-\alpha^4)\theta^4(1+o(1))}.$$
\end{proof}
By (\ref{Thomest}), in this case we have
$\sum_x|\phi^{(n)}(x)|\simeq A^nn^{1-3/4}.$
\end{exa}

\section{Bounds on convolution powers of normalized complex functions on $\bbz$}\label{sec4}

The goal of this section is to give good upper bounds for the convolution powers
$\phin $ of a given complex valued function that is finitely supported on $\mathbb Z$.
Let $\hphi$
be the Fourier transform of $\phi$ so that
$$
\phin  (x)= \frac{1}{2\pi}\int_I e^{-i x\theta}
[\hat \phi(\theta)]^nd\theta.
$$
Obviously the function $\phi$ can be normalized in some appropriate way and it is very reasonable to chose the normalization
$$
\max_{\theta\in I}\{|\hphi|\}=1.
$$
Note that this is the same as saying that the operator norm
of the convolution operator by $\phi$ acting on $\ell^2(\mathbb Z)$ is 1.

As we assume that $\phi$ is finitely supported,
either $|\hphi|=1$ on $I$, that is, $\hphi(\theta) = c e^{i j \theta}$
with $|c|=1$, $j$ an integer, or  $\max\{|\hphi|\}$
is attained at at most finitely many points in $I$ (the zeros
of the trigonometric polynomial $|\hphi|^2-1$).

In the first subsection of this section, we concentrate on the case when $\max|\hat \phi|$ is attained
at only one point $\theta_0\in I$ and $\phi(\theta)-\phi(\theta_0)$ vanishes
up to some even order.

The case when the maximum of $\hphi$ is attained at more than one point will be considered in the second subsection below but our results are much less precise in that case.

\subsection{Generalized Gaussian bounds, I}
In this section, we assume that $\max|\hat \phi|$ is attained
at only one point $\theta_0\in I$.
By replacing $\phi$ by $x\mapsto \phi_{\theta_0}(x)= \hphi(\theta_0)^{-1} e^{i x\theta_0}\phi(x)$, it is enough to consider the case when
$\theta_0=0$ and $\hat \phi(0)=1$. So assume from
now on that $\phi$ is finitely supported  with
$$
\{x: \phi(x)\neq 0\}\subset [-K,K],\;\;\|\phi\|_1=\sum|\phi|<\infty
$$
and
\begin{equation}\label{phicond}
\hphi(0)=1,\;\;|\hphi(\theta)|<1 \mbox{ for } \theta\in I^*=I\setminus \{0\}.
\end{equation}
Since  $\phi$ is finitely supported,
$\hat \phi$ is actually an entire function of the complex variable $z=u+iv\in \mathbb C$, namely,
$$
\hphi(z)=\sum_{x\in \mathbb Z} \phi(x)e^{iz x}.$$
As noticed in \cite{Thomee1,Thomee2} (and many other places),
it is most efficient to  expand
$\log \phi$ and write
$$
\hphi(\theta)= \exp( \sum_{j=1}^\infty c_j \theta ^j) \mbox{ near } 0
$$
to study the convolution powers $\phin $.
Indeed, the conditions $\hphi(0)=1$ and $|\hphi|< 1$ in $I^*$
implies that only two cases may arise. Namely, either
there exist a real $\alpha$, two  integers $0<\mu<\nu$, $\nu$ even,
a real polynomial $q$ with $q(0)\neq 0$,
and a complex number $\gamma$ with $\mbox{Re}(\gamma)>0$ such that
\begin{equation}\label{bad}
\hphi(\theta)=
e^{i\alpha \theta +i \theta^\mu q(\theta) -\gamma \theta^\nu(1+o(1))}
\end{equation}
or there exists a real $\alpha$, an even integer $\nu$
and a complex number $\gamma$ with $\mbox{Re}(\gamma)>0$ such that
\begin{equation}\label{good}
\hphi(\theta)= e^{i\alpha \theta -\gamma \theta^\nu(1+o(1))}.
\end{equation}
When \eqref{bad} holds, \cite{Thomee1} shows that
$\|\phin \|_1$ tends to infinity with  $n$ and in fact
$$
\|\phin \|_1\simeq n^{(1-\mu/\nu)/2} \mbox{ as } n\rightarrow \infty.
$$
See \cite[Sect.~7]{Thomee2} and the references therein. Hence we focus on the case when
\eqref{good} holds.
The simplest form of our main result is stated in the following theorem.
\begin{theo}
Let $\phi:\mathbb Z\rightarrow \mathbb C$
be a finitely supported complex valued function such that \eqref{good} holds true and
$|\hphi(\theta) |<1$ on $I^*$.
Then there are constants $C,c\in (0,\infty)$ such that for any $x \in\mathbb Z$ and $n\in \mathbb N^*$, we have
$$
|\phin(x) |\le \frac{C}{n^{1/\nu}}\exp\left( -c
\left(\frac{|x-\alpha n|}{n^{1/\nu}}\right) ^{\frac{\nu}{\nu-1}}\right).
$$
\label{Gauss1}
\end{theo}

\begin{proof}
Write
$$
\phin(x)= \frac{1}{2\pi}
\int_I e^{-i x\theta} \hphi(\theta)^n d\theta
= \frac{1}{2\pi}\int_I e^{-i (x -\alpha n) \theta} P(\theta) ^n d\theta
$$
where
$$
P(\theta)= e^{-i \alpha \theta} \hat \phi(\theta).
$$
Condition (\ref{good}) together with $|\hphi|<1$ on $I^*$ implies that
\begin{equation}\label{Pgood}
P(\theta)=\exp(- \gamma \theta^\nu(1+o(1))),\;\;|P|<1 \mbox{ on } I^*
\end{equation}
with $\gamma, \nu$ as in \eqref{good}, that is $\mbox{Re}(\gamma)>0$ and
$\nu$ is an even integer.  In addition $P$ is obviously
an entire function of the complex variable $z=u+iv$ and \eqref{Pgood}
implies that there are constants $\gamma_1,\gamma_2\in (0,\infty)$ such that
\begin{equation}\label{Pgood1}
|P(z)|\le \exp(-\gamma_1 u^\nu +\gamma_2 v^\nu) \mbox{ on }
\{z=u+iv: |u|\le 3\pi/2\}
\end{equation}
and
\begin{equation}\label{Pgood2}
|P(z)|\le \exp(\gamma_2 v^\nu)  \mbox{ on } \mathbb C .
\end{equation}

To see why (\ref{Pgood2}) holds, at the origin, use the assumed expansion
$P(u)=e^{i\alpha u-\gamma u^\nu(1+o(1))}$.
Away from the origin, note that
$\hphi(z)$ is a trigonometric polynomial and thus is periodic in
$u$ with growth at most exponential in $|v|$ for large $v$.
The estimate
(\ref{Pgood1}) improves significantly upon (\ref{Pgood2}) only in the sector
$$\{z=u+iv: |v|\le \frac{\gamma_2}{\gamma_1}|u|; |u|\le 3\pi/2\}$$
where the the two parts of the hypothesis \eqref{Pgood} together
give the desired result.

Next, write
$$
\phin(x)= \frac{1}{2\pi n^{1/\nu}}\int_{n^{1/\nu}I}
 e^{-i \frac{(x -\alpha n)}{n^{1/\nu}} \theta} P_n(\theta)\ d\theta
$$
with
$$ n^{1/\nu}I= (-n^{1/\nu}\pi,n^{1/\nu}\pi] \mbox{ and }
P_n(\theta)=P(\theta/n^{1/\nu})^n.
$$
It follows that
\begin{equation}\label{trouble}
\left(
\frac{|x-\alpha n|}{n^{1/\nu}}\right)^q |\phin(x)|=
\frac{1}{2\pi n^{1/\nu}}\left|\int_{n^{1/\nu}I}
 e^{-i \frac{(x -\alpha n)}{n^{1/\nu}} \theta} \partial_\theta^q
P_n(\theta)\ d\theta\right|.
\end{equation}
This step requires some explanation because when $\alpha n$ is not an integer,
neither $e^{-i \frac{(x -\alpha n)}{n^{1/\nu}} \theta}$ nor $P_n(\theta)$ are periodic function of period $2\pi n^{1/\nu}$ so that, a priori, the repeated
integration by parts used to obtain (\ref{trouble}) should produce boundary terms. However, introducing the operator $\delta$ defined on smooth functions by
$$\delta f (\theta)=  e^{-i \frac{\alpha n }{n^{1/\nu}} \theta}\partial_\theta [e^{i \frac{\alpha n }{n^{1/\nu}} \theta} f(\theta)],$$  write
\begin{eqnarray*}
\int_{n^{1/\nu}I}
[\partial_\theta^q e^{-i \frac{(x -\alpha n)}{n^{1/\nu}} \theta} ]
P_n(\theta)\ d\theta &=&
\int_{n^{1/\nu}I} e^{-i \frac{\alpha n }{n^{1/\nu}} \theta}\partial_\theta^q [e^{i \frac{\alpha n }{n^{1/\nu}} \theta}
e^{-i \frac{x}{n^{1/\nu}} \theta}]
(\hphi(\theta/n^{1/\nu}))^n \ d\theta\\
&=&   \int_{n^{1/\nu}I} \delta^r
[e^{-i \frac{x}{n^{1/\nu}} \theta}]
(\hphi(\theta/n^{1/\nu}))^n \ d\theta.
\end{eqnarray*}
Next, observe that $\delta$ preserves  the periodicity of the function $f$. In particular, if $f$ is periodic of period $2\pi n^{1/\nu}$ then $\delta f$ has the same property. Hence the formal adjoint of $\delta$ on $2\pi n^{1/\nu}$-periodic functions is given by
$$\delta^* f (\theta)= - e^{i \frac{\alpha n }{n^{1/\nu}} \theta}\partial_\theta [e^{-i \frac{\alpha n }{n^{1/\nu}} \theta} f(\theta)].$$
It follows that
\begin{eqnarray*}
 \int_{n^{1/\nu}I} \delta^r
[e^{-i \frac{x}{n^{1/\nu}} \theta}]
(\hphi(\theta/n^{1/\nu}))^n \ d\theta &= &(-1)^r\int_{n^{1/\nu}I}
e^{-i \frac{x}{n^{1/\nu}} \theta} \delta^{*r}
(\hphi(\theta/n^{1/\nu}))^n \ d\theta \\
&=&(-1)^r\int_{n^{1/\nu}I}
 e^{-i \frac{(x -\alpha n)}{n^{1/\nu}} \theta} \partial_\theta^q
P_n(\theta)\ d\theta.\end{eqnarray*}
This justifies (\ref{trouble}). We note that a formula similar to
(\ref{trouble})
appears on page 124, equation (2.2), of the classic paper
by Ney and Spitzer \cite{NS} without warning or detailed explanations.

Observe that (\ref{Pgood1})-(\ref{Pgood2}) translate immediately into
\begin{equation}\label{Pgood1n}
|P_n(z)|\le \exp(-\gamma_1 u^\nu +\gamma_2 v^\nu) \mbox{ on }
\{z=u+iv: |u|\le 3n^{1/\nu}\pi/2\}
\end{equation}
and
\begin{equation}\label{Pgood2n}
|P_n(z)|\le \exp(\gamma_2 v^\nu)  \mbox{ on } \mathbb C .
\end{equation}

The crucial estimate is given by the  following proposition.
\begin{pro}\label{pro-Gauss}
There are constants $C_1,a_1\in (0,\infty)$ such that for any $q,n=0,1,2\dots,$ and
$\theta \in n^{1/\nu}I$,
$$
|\partial_\theta^qP_{n}(\theta)|\le  C_1^{1+q} q! q^{-q/\nu}
\exp\left(-a_1 \theta^{\nu}\right) .
$$
\label{partialq}
\end{pro}

Assuming that this proposition has been proved, we obtain
$$
\left(
\frac{|x-\alpha n|}{n^{1/\nu}}\right)^q
|\phin(x)|\le n^{-1/\nu}
C_2^{1+q} q! q^{-q/\nu}.
$$
This is of the form
$$
|\phin(x)|\le C_2 n^{-1/\nu} M^{-q} q^{(1-1/\nu)q},
\;\;M= \frac{|x-\alpha n|}{ C_2n^{1/\nu}}.
$$
Elementary calculus shows that
$$
\inf_{q=0,1,2,\dots}\{M^{-q} q^{(1-1/\nu)q}\}\le C\exp(- c
M^{\frac{\nu}{\nu-1}}),\;\; c=\frac{\nu}{e(\nu-1)}.
$$
This gives the upper bound stated in Theorem \ref{Gauss1}.

It remains to prove Proposition \ref{partialq}. By Cauchy's formula,
$$
\partial^q _\theta P_n(\theta)=
\frac{q!}{2\pi i}\int_{|\xi|=r}\frac{P_n(z)}{(\zeta-\theta)^{q+1}}d\zeta, \theta\in n^{1/\nu}I.
$$
Consider two cases.

If $q\le \theta^\nu $, pick
$r= \epsilon q^{1/\nu}$ with $\epsilon>0$ small enough (depending on
$\gamma_1,\gamma_2$ in \eqref{Pgood1n}) so that
$|P_n(z)|\le \exp(- \gamma_3 \theta^\nu)$ on $|z-\theta|=r$.
This easily gives the inequality of Proposition \ref{partialq} when
$q\le \theta^\nu$.

If, instead, $q> \theta^\nu$ then pick  $r= q^{1/\nu}$ and observe that,
on $|z-\theta|=r$,
$$
|P_n(z)|\le  \exp (2\gamma_2 r^\nu)\le
\exp( - a_1 \theta^\nu + (2\gamma_2+a_1) q).$$
This yields the desired estimate when $q>\theta^\nu$.
\end{proof}

A useful complement to Theorem \ref{Gauss1} involves ``regularity'' estimates
for $\phin(x)$. Namely, for any integer $y$ and function
$f: \mathbb Z\to \mathbb C$, set $\partial_yf(x)=f(x+y)-f(x)$.
\begin{theo}
Let $\phi:\mathbb Z\rightarrow \mathbb C$
be a finitely supported function such that {\em (\ref{good})} holds true and
$|\hphi(\theta) |<1$ on $I^*$.
Then there are constants $A,C,c\in (0,\infty)$ such that
for any $x\in \mathbb Z$, $n\in \mathbb N^*$ and  any
$y_1,\dots,y_m\in \mathbb Z$ with $|y_j|\le A n^{1/\nu}$, $j=1,\dots,m$,
we have
$$
|\partial_{y_1}\cdots \partial_{y_m} \phin(x) |\le
\frac{C^m\prod_1^m |y_j|}{n^{(1+m)/\nu}}\exp\left( -c
\left(\frac{|x-\alpha n|}{n^{1/\nu}}\right) ^{\frac{\nu}{\nu-1}}\right).
$$
In particular, there exists $C_1$ such that for all $n\ge 1$ and
$y_1,\dots,y_m\in \mathbb Z$ with $|y_j|\le A n^{1/\nu}$, $j=1,\dots,m$, we have
$$\sum_{x\in \mathbb Z}|\partial_{y_1}\cdots \partial_{y_m} \phin(x) |\le
C_1\frac{C^m\prod_1^m |y_j|}{n^{m/\nu}}.$$
\label{Gauss2}
\end{theo}
\begin{proof}
Observe that
$$
\partial_{y_1}\cdots \partial_{y_m} \phin(x)
= \frac{1}{2\pi n^{1/\nu}}\int_{n^{1/\nu}I}
 e^{-i \frac{(x -\alpha n)}{n^{1/\nu}} \theta} Q_n(\theta) d\theta
$$
where
$$
Q_n(z)= \prod_{j=1}^m (e^{iy_jz/n^{1\nu}}-1)P_n(z).
$$
Clearly, for $\epsilon>0$ small enough, if $z=u+iv$
with $|u|\le 3n^{1/\nu}\pi/2$ and $|v|\le \epsilon |u|$,
\eqref{Pgood1} gives
$$
|Q_{n}(z)|\le  C_1^m\left(\prod_{j=1}^m \frac{|y_i|}{n^{1/\nu}}\right)
e^{-\gamma_4 u^\nu}.
$$
Further for all $z\in \mathbb C$, \eqref{Pgood2} gives
$$
|Q_{n}(z)|\le  C_1^m\left(\prod_{j=1}^m \frac{|y_i|}{n^{1/\nu}}\right)
e^{\gamma_5 v^\nu}.
$$
Using these estimates in place of \eqref{Pgood1}--\eqref{Pgood2}
and the line of reasoning of  the proof of Theorem \ref{Gauss1} gives the
desired estimate.
\end{proof}

We end by stating explicitly the most general result obtained in this section.

\begin{theo}
Let $\phi:\mathbb Z\rightarrow \mathbb C$
be a finitely supported function such that there exists $\theta_0\in I=(-\pi,\pi]$ such that
$$|\hphi(\theta_0)|=1 \mbox{ and }\;|\hphi(\theta)|<1 \mbox{ on } I\setminus\{\theta_0\}.$$
Assume further that there exist $\alpha\in \mathbb R$, $\gamma\in \mathbb C$ with $\mbox{Re}(\gamma)>0$, and an even integer $\nu$ such that
\begin{equation}
\hphi(\theta_0+\theta)=\hphi(\theta_0)e^{i \alpha \theta -\gamma   \theta^\nu(1+o(1))} \mbox{ at } \theta=0.
\end{equation}
Then there are constants $A,C,c\in (0,\infty)$ such that
for any $x\in \mathbb Z$, $n\in \mathbb N^*$, $m=0,1,\dots$, and
$y_1,\dots,y_m\in \mathbb Z$ with $|y_j|\le A n^{1/\nu}$, $j=1,\dots,m$,
we have
$$
|\partial_{y_1}\cdots \partial_{y_m} [ \hphi(\theta_0)^{-n}e^{ix\theta_0}\phin(x)] |\le
\frac{C^m\prod_1^m |y_j|}{n^{(1+m)/\nu}}\exp\left( -c
\left(\frac{|x-\alpha n|}{n^{1/\nu}}\right) ^{\frac{\nu}{\nu-1}}\right).
$$
Further, there exists $\epsilon,\eta>0$ such that, on $|x-\alpha n|\le \epsilon n^{1/\nu}$,
$$ \mbox{\em Re}\left(\hphi(\theta_0)^{-n}e^{ix\theta_0}\phi^{(n)}(x)\right)\ge \eta n^{-1/\nu}.$$
\label{Gauss4}
\end{theo}
\begin{proof} This follows from Theorems \ref{Gauss1}-\ref{Gauss2}
applied to  $\phi_{\theta_0}(x)=\hphi(\theta_0)^{-1}e^{ix\theta_0}\phi(x)$.
The last assertion follows from Theorem \ref{th-extLL}.
\end{proof}

\subsection{Sub-Gaussian bounds}

Examples such as \ref{3points}
show it is very natural to allow the Fourier transform
$\hphi$ to attain its maximum of $1$ at more than one point on $(-\pi, \pi]$.
This section explores briefly what global estimates can be obtained in such
cases. In the following theorem, $\hphi$ is allowed to attain its maximum, $1$,
at finitely many points. At each of this points, we assume that the dominant
non-linear term in the expansion of $\log \hphi$ is of an even degree.

\begin{theo}
Let $\phi:\mathbb Z\rightarrow \mathbb C$
be a finitely supported function such that there exists $\theta_q\in I=(-\pi,\pi]$, $q\in\{1,\dots, Q\}$ such that
$$|\hphi(\theta_q)|=1 \mbox{ and }\;|\hphi(\theta)|<1 \mbox{ on } I\setminus\{\theta_1,\dots,\theta_Q\}.$$
Assume further that for each $q\in \{1,\dots, Q\}$, there exist $\alpha_q\in \mathbb R$, $\gamma_q\in \mathbb C$ with $\mbox{Re}(\gamma_q)>0$, and an even integer $\nu_q$ such that
\begin{equation}
\hphi(\theta_q+\theta)=\hphi(\theta_q)e^{i \alpha_q \theta -\gamma_q   \theta^{\nu_q}(1+o(1))} \mbox{ at } \theta=0.
\end{equation}
Then, for each $N$, there is a constant $C_N$ such that
for any $x\in \mathbb Z$, $n=1,2,\dots$,
we have
\begin{equation}\label{estN}
|\phin(x) |\le  C_N\sum_1^Q
\frac{1}{n^{1/\nu_q}} \left(1+\frac{|x-\alpha_q n|}{n^{1/\nu_q}}\right) ^{-N}.
\end{equation}
\label{Gauss5}
\end{theo}
\begin{rem} An immediate application of Theorem \ref{Gauss5} is that
there exists a constant $C$ such that, for all $n$,
$$\sum_x|\phi^{(n)}(x)|\le C.$$
That is, we recover the positive part of the stability theorem, Theorem
\ref{oldthmA11}.  Indeed, estimates such as (\ref{estN}) are more or less implicit
in the proof of Theorem \ref{oldthmA11} given in \cite{Thomee1,Thomee2}.
\end{rem}

\begin{proof} We need to introduce smooth non-negative cut-off functions $\psi_q$ such that each $\psi_q$ vanishes outside an interval $\theta_q+I_q=[\theta_q-\epsilon_q,\theta_q+\eta_q]$, $\psi_q\equiv 1$
on $[\theta_q-\epsilon_q/2,\theta_q+\eta_q/2]$,
$\theta_{q'}\not\in [\theta_q-3\epsilon_q/2,\theta_q+3\eta_q/2]$, $q'\neq q$, and  $\sum_1^Q \psi_q\equiv 1$ (i.e., the $\psi_q$'s form a partition of unity). Using these cut-off functions, we write
\begin{eqnarray*}
\phin(x)&=& \frac{1}{2\pi}
\int_I e^{-i x\theta} \hphi(\theta)^n d\theta
= \sum_1^Q\frac{1}{2\pi}\int_{I} e^{-ix \theta} \hphi(\theta)^n \psi_q (\theta) d\theta\\
&=& \sum_1^Q\frac{e^{-i x\theta_q}\hphi(\theta_q)^n}{2\pi}\int_{I} e^{-i(x-\alpha_qn) \theta}[ \hphi(\theta_q)^{-1}e^{-\alpha_q \theta} \hphi(\theta_q+\theta)]^n \psi_q(\theta_q+\theta)d\theta\\
&=&\sum_1^Q\frac{e^{-i x\theta_q}\hphi(\theta_q)^n}{2\pi}\int_{I_q} e^{-i(x-\alpha_qn) \theta} P_q(\theta)^n \psi_q(\theta_q+\theta) d\theta\\
&=&\sum_1^Q\frac{e^{-i x\theta_q}\hphi(\theta_q)^n}{2\pi n^{1/\nu_q}}\int_{n^{1/\nu_q}I_q} e^{-i \frac{(x-\alpha_qn)}{n^{1/\nu_q}} \theta} P_{q,n}(\theta)\psi_{q,n}(\theta)d\theta
\end{eqnarray*}
where
\begin{eqnarray*}
P_{q}(\theta)&=&\hphi(\theta_q)^{-1}e^{-\alpha_q \theta} \hphi(\theta_q+\theta),\\
P_{q,n}(\theta) &=& P_q(\theta/n^{1/\nu_q})^n,\\
\psi_{q,n}(\theta)&=&\psi_q(\theta_q+\theta/n^{1/\nu_q}).
\end{eqnarray*}
By hypothesis,
$$P_q(\theta)= e^{-\gamma \theta^{\nu_q}(1+o(1))}\,\,\mbox{ at } 0.$$
The function $P_q$ can be viewed as an entire function of
$z=u+iv\in \mathbb C$ which is periodic in $u$ and has at most exponential growth in $|v|$ for large $|v|$. Hence there are constants $\gamma_1,\gamma_2\in (0,\infty)$ such that
\begin{equation}\label{Pgood1q}
|P_q(z)|\le \exp(-\gamma_1 u^{\nu_q} +\gamma_2 v^{\nu_q}) \mbox{ on }
\{z=u+iv: u \in (-3\epsilon_q/2,3\eta_q/2) \}
\end{equation}
and
\begin{equation}\label{Pgood2q}
|P_q(z)|\le \exp(\gamma_2 v^{\nu_q})  \mbox{ on } \mathbb C .
\end{equation}
A priori, the constants $\gamma_1,\gamma_2$ depend on $q$ but since there are only finitely many $q$, one can assume that they are the same for all $q$.

Next, for each $q\in \{1,\dots, Q\}$, consider
$$\mathcal I_q= \frac{1}{2\pi n^{1/\nu_q}}\int_{n^{1/\nu_q}I} e^{-i \frac{(x-\alpha_qn)}{n^{1/\nu_q}} \theta} P_{q,n}(\theta)\psi_{q,n}(\theta)d\theta.$$
Observe that
\begin{eqnarray}
\left(-i\frac{(x-\alpha_q n)}{n^{1/\nu_q}}\right)^k \mathcal I_q&=&
\frac{1}{2\pi n^{1/\nu_q}}\int_{n^{1/\nu_q}I}\partial_\theta^k \left[e^{-i \frac{(x-\alpha_qn)}{n^{1/\nu_q}} \theta}\right] P_{q,n}(\theta)\psi_{q,n}(\theta)d\theta \nonumber\\
&=&
\frac{1}{2\pi n^{1/\nu_q}}\int_{n^{1/\nu_q}I_q} e^{-i \frac{(x-\alpha_qn)}{n^{1/\nu_q}} \theta} \partial^k_\theta
[P_{q,n}(\theta)\psi_{q,n}(\theta)] d\theta.  \label{split}
\end{eqnarray}
The iterated integration by parts performed to obtain the second equality does not produce boundary terms thanks to the cutoff function $\psi_{q,n}$. The following proposition is analogous to Proposition \ref{pro-Gauss} and the proof is the same.
\begin{pro}\label{pro-Gaussq}
There are constants $C_1,a_1\in (0,\infty)$ such that for any $q,n=0,1,2\dots,$ and
$\theta \in  (-n^{1/\nu_q}\epsilon_q,n^{1/\nu_q}\eta_q)$,
$$
|\partial_\theta^kP_{q,n}(\theta)|\le  C_1^{1+k} k! k^{-k/\nu}
\exp\left(-a_1 \theta^{\nu_q}\right) .
$$
\end{pro}
Using (\ref{split}) and Proposition \ref{pro-Gaussq}, we obtain
\begin{eqnarray}
\left|\frac{(x-\alpha_q n)}{n^{1/\nu_q}}\right|^k \mathcal I_q&\le&
\frac{C_k}{ n^{1/\nu_q}}.\end{eqnarray}
The desired result follows. The necessity to separate the contributions of the different $\theta_q$  via the use of cutoff functions prevents us to obtain
more precise upper bound in this case.\end{proof}

\section{Carne's transmutation formula}\label{sec5}

This section develops a version of Carne's transmutation formula to
obtain universal long range upper bound on $M_k^n(x,y)$ where
$M$ is a (finite range)  normal contraction  acting on $\ell^2(X,\pi)$.
Here, $X$ is a countable space,  $\pi$ a positive measure, $M$ acts on
$\ell^2(X,\pi)$ by
$$
Mf(x)=\sum_y M(x,y)f(y),
$$
$k$ is a fixed integer and
$$
M_k=I-(I-M)^k=\sum (-1)^j\binom {k}{j} M^j.
$$
If we think of $I-M$ as a "Laplacian", then $M_k^n$ is the discrete semigroup of operators associated with the $k$-th power of this Laplacian, namely,  $(I-M)^k$.

Assume that $X$ is equipped with a metric $d$ such that
\begin{equation}\label{FR}
d(x,y)>1 \Longrightarrow M(x,y)=0.
\end{equation}
This is a finite range condition in terms of the metric $d$. If the matrix $M$
is finite range in the sense that for each $x$ there are finitely many
$y$ such that $M(x,y)\neq 0$ then we can define the metric $d_M$ by
$$d_M(x,y)=\inf\{ k: M^k(x,y)\neq 0\}.$$
This metric obviously has property (\ref{FR}) (we do not have to assume
that $d_M$ is finite for all $x,y$).

Our goal is to obtain bounds
showing that $M_k^n(x,y)$ is small when $d(x,y)$ is large when compared to $n^{1/(2k)}$.
In the classical case where $k=1$ and $M$ is the transition matrix of
a reversible Markov chain, the Carne-Varopoulos estimate states
that
$$
|M^n(x,y)|\le 2\left(\frac{\pi(x)}{\pi(y)}\right)^{1/2}
\exp\left(-\frac{d(i,j)^2}{2n}\right).
$$
This bound is remarkable for its generality and explicitness. To see it in action for Markov chains on $\bbz^d$,  see \cite{barlow}. For  an extension to non-reversible chains, see \cite{Mathieu}. For a probabilistic interpretation, see \cite{peyre}.

Our result reads as follows.
\begin{theo} Let $(X,d)$ be a countable metric space equipped with a positive
measure $\pi$. Fix $k$ and $s_0\in (0,2^{-1+1/k})$.
There exist constants $C,c\in (0,\infty)$ (depending on $k$ and $s_0$) such that, for any normal contraction $M$  on $\ell^2(X,\pi)$ satisfying {\em (\ref{FR})} and  whose spectrum is contained in $[a,1]$
with $a\in [-1,1)$ and $1-a \le 2s_0$,
$$\forall x,y\in X,\;n=1,\dots,\;\;
|M_k^n(x,y)|\le C\left(\frac{\pi(x)}{\pi(y)}\right)^{1/2}
\exp\left(- c\left(\frac{d(x,y)}{n^{1/(2k)}}\right)
^{\frac{2k}{2k-1}}\right).
$$
\label{th-Carne}
\end{theo}
\begin{rem} Proving that an operator is normal and computing
its spectrum is not an easy task, in general. Most application of
Theorem \ref{th-Carne} are likely to involve cases when $M$ is a
hermitian contraction. In this case the hypothesis that the
spectrum is contained in $[a,1]$ is a very natural hypothesis since the spectrum is real and contained in $[-1,1]$.
\end{rem}

\begin{proof}
Following \cite{Carne}, consider the Chebyshev polynomials
$$
Q_m(z)=\frac{1}{2}\left((z+(z^2-1)^{1/2})^m+(z-(z^2-1)^{1/2})^m\right),
\;\;m\in \mathbb Z.
$$
Each $Q_m$ is in fact a polynomial of degree $|m|$.
For $a\in [-1,1]$, let
$$
Q_{a,m}(z)=Q_k\left(\frac{2z-1-a}{1-a}\right).
$$
These are the Chebyshev polynomials for the interval $[a,1]\subset [-1,1]$.
See \cite[Theorem 2']{Carne}.
As explained in \cite[Sect.~2]{Carne}, assuming that the spectrum of $M$
on $\ell^2(X,\pi)$ is contained in $[a,1]$, it holds that
$$
Q_{a,m}(M):\ell^2(X,\pi)\to\ell^2(X,\pi) \mbox{ is a contraction}
$$
and
\begin{equation}\label{Carne}
M^n=\sum_m\mathbf P_{0}(X^a_n=m)Q_{a,m}(M)
\end{equation}
where $X^a_n$ is the simple random walk on $\mathbb Z$ with transition
probabilities
$$
\mathbb P(X_n=\pm 1|X_{n-1})=(1-a)/4,\;\;\mathbb P(X_n=0|X_{n-1})=(1+a)/2.
$$
Consider the measures $\beta_s$ on $\mathbb Z$  where $\beta$ is the
Bernoulli measure $\beta(\pm 1)=1/2$ and
$$
\beta_s=(1-s)\delta_0+s\beta, \;\;s\in [0,1].
$$
Further, set
$$
\Delta_sf= f*(\delta_0-\beta_s),\;\;
\psi_{s,k}= \delta_0-(\delta_0-\beta_s)^{*k}.
$$
By definition,
$$
\mathbb P_0(X^a_n=m)= \beta_{s}^{(n)}(m),\;\;s=(1-a)/2.
$$
Also, by \eqref{Carne},
\begin{equation}\label{Carnek}
M_k^n=(I-(I-M)^k)^n= \sum_{m\in \mathbb Z} \psi_{s,k}^{(n)}(m)Q_{a,m}(M).
\end{equation}
Similarly,
\begin{equation}\label{Carnekl}
(I-M)^\ell M_k^n=
\sum_{m\in \mathbb Z} \Delta_{s}^\ell\psi_{s,k}^{(n)}(m)Q_{a,m}(M).
\end{equation}
Now, let $\phi_i$, $i=1,2$, be two functions on $X$ with support in
$B_i=B(x_i,r_i)$ and assume that $d(B_1,B_2)\ge r$. Then
$$
\langle (I-M)^{\ell}M_k^n \phi_1,\phi_2\rangle_\pi
= \sum_{m\in \mathbb Z} \Delta_{s}^\ell\psi_{s,k}^{(n)}(m)\langle
Q_{a,m}(M) \phi_1,\phi_2\rangle_\pi.
$$
Further, since $Q_{a,m}(M)$ is an $\ell^2(X,\pi)$ contraction and a polynomial in $M$ of degree $|m|$,
$$
|\langle Q_{a,m}(M) \phi_1,\phi_2\rangle_\pi| \le \|\phi_1\|_2\|\phi_2\|_2
$$
and
$$
\langle Q_{a,m}(M) \phi_1,\phi_2\rangle_\pi=0 \mbox{ if } |m|< r.
$$
This last property follows from the hypothesis that $M(x,y)=0$ if $d(x,y)>1$ which
implies $M^i(x,y)=0$ is $d(x,y)>i\ge 0$.
Putting these properties together, we obtain
$$
|\langle (I-M)^{\ell}M_k^n \phi_1,\phi_2\rangle_\pi|
\le  \|\phi_1\|_2\|\phi_2\|_2\sum_{|m|\ge r} |\Delta^\ell\psi_{s,k}^{(n)}(m)|,\;
s=(1-a)/2.
$$

For $0<s=(1-a)/2\le s_0<2^{-1+1/k}$,  the Fourier transform
$$
\hat\psi_{s,k}
(\theta)= 1-s^k(1-\cos \theta)^k)
$$
satisfies
$|\hat \psi_{s,k}|<1 $ on $I^*$ and \eqref{good} with $\nu=2k$ and
$\gamma= (s/2)^k$. One can conclude that there are constants
$C,c\in (0,\infty)$ depending only on $s$ such that
\begin{equation}
|\langle (I-M)^{\ell}M_k^n \phi_1,\phi_2\rangle_\pi|
\le \frac{C^\ell  \|\phi_1\|_2\|\phi_2\|_2}{(1+n)^{\ell/k}} \exp\left(- c \left(\frac{r}{n^{1/(2k)}}\right)^{\frac{2k}{2k-1}}\right).
\label{oldA34}
\end{equation}
The inequality stated in Theorem \ref{th-Carne} follows by taking
$\phi_1=\delta_{x}$, $\phi_2=\delta_{x_2}$, $x_1=x,x_2=y$, $r_1,r_2=0$,
$r=d(x,y)$.
\end{proof}

Inequality (\ref{oldA34}) contains more information than captured in Theorem \ref{th-Carne}. In particular, we get the following result.

\begin{theo}
Let $(X,d)$ be a countable metric space equipped with a positive measure $\pi$. Fix $k$ and
$s_0\in (0,2^{-1+1/k})$.
There exist constants $C,c\in (0,\infty)$ (depending on $k$ and $s_0$) such that, for any normal contraction $M$  on $\ell^2(X,\pi)$ satisfying {\em (\ref{FR})} and  whose spectrum is contained in $[a,1]$
with $a\in [-1,1)$ and $1-a \le 2s_0$, we have
$$
|(I-M)^\ell M_k^n(x,y)|\le C^{\ell}\left(\frac{\pi(x)}{\pi(y)}\right)^{1/2}
(1+n)^{-\ell/k}\exp\left(- c\left(\frac{d(x,y)}{n^{1/(2k)}}\right)
^{\frac{2k}{2k-1}}\right).
$$
Further, if $A_x(r,R)=\{y: r < d(x,y)\le R\}$ then
$$\sum_{y\in A_x(r,R)}|(I-M)^\ell M_k^n(x,y)|^2\pi(y)\le \frac{C^{\ell}\pi(x)^{1/2}}{(1+n)^{\ell/k}}
\exp\left(- c\left(\frac{r}{n^{1/(2k)}}\right)
^{\frac{2k}{2k-1}}\right).$$
\label{th-Carne2}
\end{theo}

Although these results are far from optimal in the sense
that they do not capture the decay of $M_k^n(x,x)$,
they do contain some useful information as demonstrated in the following
corollary which provide a highly non-trivial lower bound on $M^{2n}_k(x,x)$. This follows closely one of the result obtained in \cite{CG}
in the reversible Markov case and with $k=1$.

\begin{theo}
Let $(X,d)$ be a countable metric space equipped with a positive measure $\pi$.
Fix $k$ and $s_0\in (0,2^{-1+1/k})$.
Let $M$ be a hermitian contraction satisfying {\em (\ref{FR})},
$\sum_y M(x,y)=1$ and with spectrum contained in $(-2s_0+1,1]$.
Fix $x\in X$, set  $V(x,t)=\pi(B(x,t))$
 and assume that there exists a positive
increasing
function $v$ such that
\begin{equation}\label{Car1}
V(x,2t)\le \pi(x)v(t), \;\;v(0)\ge (1+2C)^{1/2},\end{equation}
and with
\begin{equation}\label{Car2}
t\mapsto t^{\frac{2k}{2k-1}}/\log v (t) \mbox{ increasing to infinity. }
\end{equation}
Given $n$, let $r(n)$ be the smallest integer such that
\begin{equation}\label{Car3}
5 n^{1/(2k-1)}\le c r^{\frac{2k}{2k-1}}/\log v(r).
\end{equation}
Then
$$
M_k^{2n}(x,x)\ge \frac14\frac{\pi(x)}{V(x,r(n))}.
$$
In particular,
\begin{enumerate}
\item If $V(x,t)\le A\pi(x)(1+ t)^D$ then
$$
M_k^{2n}(x,x)\ge a \pi(x) [(1+n)\log(1+ n)]^{-D/2k}.
$$
\item If $V(x,t)\le \pi(x) \exp( A t^\beta)$ with $\beta\in (0,\frac{2k}{2k-1})$,
$$
M^{2n}_k(x,x)\ge \pi(x)\exp\left(- A_0 n^{\frac{\beta(2k-1)}{2k(1-\beta)+\beta}}\right).
$$
\end{enumerate}
\label{oldthmA33}
\end{theo}
\begin{rem}
The last inequality is informative only when $\beta< 1/(2k-1)$ because
$M^n_k(x,x)\ge \pi(x) V(x,n)^{-1}$, always.
\end{rem}
\begin{proof}
Set $B(x,r)=\{z:d(x,z)\le r\}$
and $m_k^n(x,y)=M_k^n(x,y)/\pi(y)$. Since $M$ is hermitian, that is, $M(x,y)/\pi(y)=\overline{M(y,x)}/\pi(x)$, and
$\sum_y M(x,y)=1$, we have
(the first step uses that $M$ is hermitian)
\begin{align*}
m_k^{2n}(x,x) &= \sum_z |m_k^n(x,z)|^2\pi(z)
\ge \sum_{z\in B(x,r)} |m_k^n(x,z)|^2\pi(z)\\
&\ge \frac{1}{\pi(B(x,r))} \left|\sum_{z\in B(x,r)} m_k^n(x,z)\pi(z)\right|^2\\
&\ge \frac{1}{\pi(B(x,r))}
\left(1- \sum_{z\in B(x,r)^c} |m_k^n(x,z)|\pi(z)\right)^2.
\end{align*}
Next, set $A_q=A(r2^q,r2^{q+1})$ and, using Theorem \ref{th-Carne2},  write
\begin{align*}
\sum_{z\in B(x,r)^c} |m_k^n(x,z)|\pi(z)&=\sum_{q=0}^\infty \sum_{z\in A_q} |M_k^n(x,z)|\\
&\le C \pi(x)^{1/2} \sum_q\pi(A_q)^{1/2}e^{- c(2^qr)^{\frac{2k-1}{2k}}/n^{1/(2k-1)}}\\
&\le C \pi(x)^{1/2} \sum_q e^{- c\frac{(2^qr)^{\frac{2k}{2k-1}}}{n^{1/(2k-1)}}+\frac12\log
    [\frac{\pi(B(x,2^{q+1}r))}{\pi(x)}]}\\
\end{align*}
Using the  function $v$ given by (\ref{Car1})-(\ref{Car2}) and any $r$
such that $$
5 n^{1/(2k)} \le c r^{\frac{2k}{2k-1}}/\log v(r),
$$
the first term in the series above is bounded by
$$
e^{- c\frac{r^{\frac{2k}{2k-1}}}{n^{1/(2k)}}+\frac{1}{2}
\log [\frac{\pi(B(x,2r))}{\pi(x)}]}\le e^{- (9/2)\log v(r) }.
$$
Further the ratio of two consecutive terms in the series is bounded by
$$
e^{-\frac{c}2\frac{(2^{q+1}r)^{\frac{2k}{2k-1}}}{n^{1/(2k)}}+\frac12\log v(2^{q+1}r)}\le e^{-2\log v(r)}.
$$
Hence,
\begin{align*}
\sum_{z\in B(x,r)^c} |m_k^n(x,z)|\pi(z) &\le
     \frac{C e^{-(9/2) \log v(r) +2\log v(r)}}{e^{2\log v(r)}-1} \\
&\le \frac{C}{v(r)^{1/2}(v(r)^2-1)}\le \frac{C}{v(r)^{1/2}(v(0)^2-1)}
 \le \frac{1}{2}.
\end{align*}
It follows that
$$
M_k^{2n}(x,x)\ge \frac14\frac{\pi(x)}{\pi(B(x,r(n))}
$$
where $r(n)$ is the smallest integer such that
$$
5 n^{1/(2k)} \le c r^{\frac{2k}{2k-1}}/\log v(r).
$$

\end{proof}

\section{Applications and history}\label{sec6}

Iterations of signed kernels are useful in data smoothing, image analysis, density estimation, and summability. There are a variety of applications to partial differential equations via divided difference schemes for numerical computation and asymptotics of higher-order equations. These applications served as motivation behind most of the previously developed theory. In this section, we briefly give pointers
to the literature and describe the basic ideas behind these applications.

\subsection{Data smoothing, de Forest's problem, and density estimation}\label{sec21}

If $y_1,\dots,y_n,\dots$ is a sequence of numbers, it is standard practice to smooth this sequence by local averaging to eliminate noise. An early instance of this, due to C.S.\ Pierce, replaced $y_i$ by $y_i'=(y_i+y_{i+1})/2$. This was iterated four times giving $y_i^{(4)}=2^{-4}\sum_{j=1}^4\binom{4}{j}y_{i+j}$. Because the binomial distribution
approximates a normal distribution, this is an early version of smoothing with a Gaussian kernel. This story is told by Stigler \cite{Stigler78}. The next few paragraphs constitute a review of some more sophisticated examples.

\subsection*{de Forest's problem}
It is natural to insist that a smoothing method applied to a constant sequence returns the same constant sequence. For convolution smoothers
$$ y'_i=\sum f(i-j) y_j $$ (applied to two-sided sequences) this forces $\sum f(i)=1$. In early work on smoothers, de Forest asked about smoothers that would preserve low order polynomials. To achieve this, one must allow $f$ to change sign. For example, convolution by $f$ supported on $\{0,\pm 1,\pm 2\}$ with $f(\pm 2)=-1/9$, $f(\pm 1)=4/9$ and $f(0)=1/3$ preserves all quadratic polynomials. de Forest's problem refers to the study of the behavior of the iteration of such smoothers. See \eqref{24} below for this example. A charming history of de Forest with good references is in \cite{Stigler} and hard to find papers are reprinted in \cite{Stigler78}. Schoenberg \cite{Scho-Bull} studies de Forest's problem melding it with a discussion of variation diminishing and total positivity. This work is followed up by Greville \cite{Grev1,Grev2}.

\subsection*{Other smoothing techniques} In recent years, a host of techniques using iteration and negative weights have emerged. An early iterative technique called ``3RSSH'' smooths a series by replacing $y_i$ by the median of $y_{i-1}, y_i, y_{i+1}$ \cite{Tukey}. This is iterated until it is stable. Tukey introduced a smoothing method called ``twicing'' (see, e.g., \cite[Sect. 11]{Berlinet} and \cite{Devroye,Stuetzle}). This is  a linear method in which a series $y_i$ is smoothed with a kernel $K$ and then the residuals $y_i-y'_i$ are smoothed by $K$ and added in. This amounts to smoothing with $2K-K*K$. A further iteration leads to $3K-3K*K +K*K*K$. More generally, $I-(I-K)^{*k}$ gives the $k$-fold iterate. This fits exactly into some of questions discussed in Proposition \ref{prop-B}.
Indeed, Tukey's smoothing is based on the Bernoulli measure $\beta$
(with no holding). Proposition \ref{prop-B} shows that the $n$-th convolution
power of  $\phi=2\beta-\beta^{(2)}$ blows up with $n$. The Bernoulli measure
$\beta_{1/2}$ (which has holding equals to $1/2$) provides a measure such that
the convolution powers $\phi_k^{(n)}$ of
$\phi_k=\delta_0-(\delta_0-\beta_{1/2})^k$ behaves well for all $k$.

\subsection*{Density estimation} Given a random sample $y_1,\dots,y_n$ from an unknown probability density $g$, the kernel density estimator of $g$ is $$ \hat{g}(x)= \frac{1}{nh}\sum_1^nK\left(\frac{x-y_i}{h}\right). $$ A huge statistical literature studies the properties of such estimators, their behavior when $n$ is large, choice of $h$ and choice of the kernel $K$. A standard choice is the Gaussian kernel $K(x)=\frac{1}{\sqrt{2\pi}}e^{-x^2/2}$. A recurring theme known variously as ``higher-order kernels'' or ``super kernels'' allows kernels that change sign. A kernel is of order $k$ if $$ \int x^mK(x)dx=\begin{cases}1&\text{if }m=1\\0&\text{if }m=1,\dots,k-1\\c&\text{if }m=k\quad(c\neq0).\end{cases} $$ Equivalently (assuming $K$ has high moments), the Fourier transform $\hat{K}$ satisfies $\hat{K}(\xi)=1 +c\xi^k +o(|\xi|^{k})$. The paper \cite{HallM} develops these ideas in modern language and focuses on $$ K(x)= (2\pi)^{-1}\int \cos(tx) e^{-|t|^k}dt,\qquad\hat{K}(\xi)=e^{-|\xi|^k},\qquad k\geq2. $$  For even $k$, these are exactly the de Forest--Schoenberg--Greville limiting kernels. Many variations of these ideas have been explored and \cite{Berlinet,DevLug} are good sources for this material. The associated density estimators have better mean-square error but, because higher order kernels must have negative values, may give negative estimators for $g$.

\subsection{Partial differential equations}\label{sec22}

\subsubsection*{Higher-order evolution equations}

Natural scientific problems give rise to higher-order equations such as the initial value problem of solving for $u(x,t)$ satisfying
\begin{equation*}
\frac{\partial u}{\partial t}= Hu,\qquad\text{given $u(x,0)=v(x)$ for $x\in\real^n$},
\end{equation*}
with $H$, say,  a uniformly elliptic operator of order $2m$ densely defined on $L^2(\real^n)$. Under conditions on $H,\ u(t,x)$ can be expressed as $$u(t,x)=(e^{Ht}v)(x)=\int_{\real^n}K(t,x,y)v(y)\,dy.$$ If $H$ has constant coefficients, $K(t,x,y)=P_t(x-y)=\frac1{(2\pi)^n}\int_{\real^n}e^{i(x-y)\xi-p(\xi)t}\,d\xi$ with the polynomial $p(\xi)$ the symbol of $H$. In the special case that $p(\xi)$ is homogeneous of order $2m$ (so $p(a\xi)=a^{2m}p(\xi)$), $P_t(x)=t^{-n/2m}P(xt^{-1/2m})$ with $P=P_1$. Thus estimates on $P(x)=\frac1{(2\pi)^n}\int_{\real^n}e^{ix\xi-p(\xi)}\,d\xi$ determine the long-term behavior of $K$ and so $u(t,x)$. This program is explained and carried out in a series of papers by Brian Davies and coauthors \cite{Davsharp,Davhigh,Davlong,DavLp,Davlow}.

Consider the case where $n=1$ and $Hu=\frac{\partial^{2k}}{\partial x^{2k}}u$. Then $p(\xi)=\xi^{2k}$ and $P(x)=\frac1{2\pi}\int_{\real}e^{ix\xi-\xi^{2k}}\,d\xi$. This is exactly $H_{2k}(x)$ of \eqref{13}. It is well known that

\begin{equation*}
H_{2k}(x)=\frac1{2kn}\sum_{n=0}^\infty\frac{(-x)^2}{(2n)!}\frac{\Gamma(2n+1)}{2k},\quad H_{2k}(0)=\frac{\Gamma\left(\frac1{2k}\right)}{2k\pi},\quad
\int_{-\infty}^\infty H_{2k}(x)=1.
\end{equation*}
Further there exists $a,b,c\in(0,\infty)$ such that for all $m=1,2,3,\dots$ and real $u$,
\begin{equation*}
\left|\left(\frac{d}{du}\right)^mH_{2k}(u)\right|\leq c\left(\frac{m}{b}\right)^m e^{-a|u|^{2k/(2k-1)}}.
\end{equation*}
See \cite[Chap.~4]{gelfand}. This is also true (with $0^0=1$) with $m=0$. Pictures of $H_2$ and $H_4$ are in \ref{fig1} of \ref{sec-int}.

\subsubsection*{Divided difference schemes}

Iterated convolutions of complex kernels arise in finite difference approximations to the solutions of partial differential equations. Consider first the simple equation
\begin{equation}
\frac{\partial u}{\partial t}=a\frac{\partial u}{\partial x};
 \qquad0\leq t\leq T,\quad0\leq x\leq1,\quad u(x,0)=v(x)\text{ given}.
\label{21}
\end{equation}
In \eqref{21} $a$ is a given constant and $u(x,t)$ is to be found given $v$. A finite difference scheme for this equation discretizes time into multiples of $\Delta t$ and space into multiples of $\Delta x$. If $x_i=i\Delta x,\ t_j=j\Delta t$, one simple scheme is
\begin{equation}
\frac{u(x_i,t_{j+1})-u(x_i,t_j)}{\Delta t}=\frac{a(u(x_{i+1},t_j)-u(x_i,t_j))}{\Delta x}
\label{22}
\end{equation}
or, with $\lambda=\Delta t/\Delta x$,
\begin{equation}
u(x_i,t_{j+1})=(1-a\lambda)u(x_i,t_j)+a\lambda u(x_{i+1},t_j)=\phi^{(j)}\ast v(x_i)
\label{23}
\end{equation}
with $\phi(0)=(1-a\lambda),\ \phi(-1)=a\lambda$. Thus, ``running'' the differential equation corresponds to repeated convolution of the (perhaps complex) probability measure $\phi$. Higher-order constant coefficient equations and more sophisticated difference schemes lead to more complex functions $\phi$. A readable introduction to finite difference equations is in \cite{quart}.

One may approximate even-order derivatives by symmetric differences; writing $h=\Delta x$,
\begin{align*}
f^{(2)}(x)&\doteq\frac1{h^2}\left\{f(x-h)-2f(x)+f(x+h)\right\}
  =\frac2{h^2}\left\{\frac{f(x-h)}2-f(x)-\frac{f(x+h)}2\right\}\\
f^{(4)}(x)&\doteq\frac1{h^4}\left\{f(x-2h)-4f(x-h)+6f(x)-4f(x+h)+f(x+h)\right\}\\
f^{(2m)}(x)&=\frac1{h^{2m}}\sum_{j=0}^{2m}(-1)^j\binom{2m}{j}f(x+(j-m)h).
\end{align*}
Similarly, when $\mu(1)=\mu(-1)=\frac12,\ \mu(x)=0$ otherwise, we have

\begin{center}
\begin{tabular}{|c|c|c|c|}
  \hline
 $ x $& $-1$ & $0$ & $1 $\\ \hline
 $ \delta_0-\mu $& $-\frac12$ & $1 $& $-\frac12$ \\
  \hline
\end{tabular} \;\;and \;\;
\begin{tabular}{|c|c|c|c|c|c|}
  \hline
  $x$ & $-2$ & $-1$ & $0$ & $1$ & $2$ \\\hline
  $ (\delta_0-\mu)^{(2)} $ & $\frac14$ & $-1$ & $\frac32$ & $-1$ & $\frac14$ \\
  \hline
\end{tabular}
\end{center}

Thus,
\begin{equation*}
f^{(2)}=\frac{-2}{h^2}(I-\mu)\ast f,\quad f^{(4)}=\frac{2^2}{h^4}(I-\mu)^{(2)}\ast f,\dots,f^{(2m)}
 =\frac{(-2)^m}{h^m}(\delta_0-\mu)^{(m)}\ast f.
\end{equation*}
Now consider the equation
\begin{equation}
\frac{\partial u}{\partial t}=\frac{\partial^{2m}}{\partial x^{2m}}u\qquad\text{with }u(x,0)=v(x).
\label{24}
\end{equation}
Discretizing,
\begin{equation*}
\frac{u(x_i,t_{j+1})-u(x_i,t_j)}{\Delta t}=\frac{(-2)^m}{h^m}(\delta_0-\mu)^{(m)}\ast v
\end{equation*}
or
\begin{equation*}
u(x_i,t_{j+1})=\left(\delta_0+\lambda(\delta_0-\mu)^{(m)}\right)\ast
u(x_i,t_j),\qquad\lambda=(-2)^m\Delta t/h^m.
\end{equation*}
If $h=2(\Delta t)^{1/m}$, this gives again $\delta_0-(\delta_0-\mu)^{(m)}=\phi$, thus again, running the differential equation \eqref{24} corresponds to repeated convolution by the signed probability $\phi$.

A more sophisticated discussion of these ideas together with assorted results
can be found in Thom\'ee's survey \cite{Thomee2}.

\def\cprime{$'$}

\section*{Acknowledgement}
We thank James Zhao for the useful \ref{fig1} and Marty   Isaacs  for useful conversations.

\end{document}